\documentclass[12pt,oneside,a4papre]{amsart}

\usepackage{amssymb}
\usepackage{amsmath}
\usepackage{amsthm}
\usepackage{amscd}

\usepackage{longtable}
\usepackage{mathrsfs}

\usepackage[dvipdfmx]{xcolor}
\usepackage[dvipdfmx]{pict2e}
\usepackage[dvipdfmx]{graphicx}
\usepackage{comment}
\usepackage{todonotes}

\usepackage{tikz}

\usepackage[all]{xy}

\usetikzlibrary{cd}

\theoremstyle{plain}
\newtheorem{theorem}{Theorem}[section]
\newtheorem{lemma}[theorem]{Lemma}
\newtheorem{corollary}[theorem]{Corollary}
\newtheorem{proposition}[theorem]{Proposition}
\newtheorem{remark}[theorem]{Remark}
\newtheorem{question}[theorem]{Question}

\theoremstyle{definition}
\newtheorem{definition}[theorem]{Definition}

\newcommand{\kah}{K\"{a}hler }
\newcommand{\idd}{i\partial\overline{\partial}}

\subjclass[2020]{32L10, 32G05, 14F18}
\keywords{ 
direct image sheaves, Nakano positivity, singular Hermitian metrics, $L^2$-estimates.}

\begin{document}
\title
[Positivity of direct image sheaves]
{Dual Nakano positivity and singular Nakano positivity of direct image sheaves}
\author{Yuta Watanabe}
\date{}


\begin{abstract}
    Let $f:X\to Y$ be a surjective projective map and $L$ be a holomorphic line bundle on $X$ equipped with a (singular) semi-positive Hermitian metric $h$.
    In this article, by studying the canonical metric on the direct image sheaf of the twisted relative canonical bundles $K_{X/Y}\otimes L\otimes\mathscr{I}(h)$, 
    we obtain that this metric has dual Nakano semi-positivity when $h$ is smooth and there is no deformation by $f$
    and that this metric has locally Nakano semi-positivity in the singular sense when $h$ is singular. 
\end{abstract}

\vspace{-5mm}

\maketitle


\vspace{-8mm}

\section{Introduction}

Let $X$ be a \kah manifold of dimension $m+n$ and $Y$ be a complex manifold of dimension $m$.
We consider a proper holomorphic submersion $f:X\to Y$. The relative canonical bundle $K_{X/Y}$ corresponding to the map $f$ is $K_{X/Y}=K_X\otimes f^*K_Y^{-1}$.
There is a natural isomorphism $K_{X/Y}|_{X_t}\cong K_{X_t}$ when restricted to a generic fiber $X_t$ of $t\in Y$.
It is effective in many studies that the variation of the complex structure of each fiber $X_t$ is reflected in the positivity of the relative canonical bundle $K_{X/Y}$.
Therefore, the positivity properties of this bundle play important role in the study of the several complex variables and complex algebraic geometry.
In practice, we frequently deal with twisted versions $K_{X/Y}\otimes L$, where $L\to X$ is a holomorphic line bundle equipped with a smooth (semi)-positive Hermitian metric $h$.
One way to research the properties of this bundle is the direct image sheaf $f_*(K_{X/Y}\otimes L)$ on $Y$.

The positivity of this direct image sheaf has been well studied in \cite{Ber09}, \cite{BP08}, \cite{PT18}, \cite{HPS18}, \cite{BPW19}, \cite{DNWZ22}. 
In \cite{Ber09}, Berndtsson showed that the smooth canonical Hermitian metric $H$ induced by $h$ has Nakano (semi)-positivity (see Theorem\,\ref{Ber09, Thm1.2}).
First, we show that the smooth canonical Hermitian metric $H$ has dual Nakano (semi)-positivity if complex structures of fibers has no variation, this means that we can take the Kodaira-Spencer forms to be zero.
Introducing the $(n-1,n-1)$-form to determine dual Nakano positivity (see Definition \ref{def of T for dual Nakano}), we prove it by taking over Berndtsson's method of calculation to compute the positivity of curvature.

\begin{theorem}\label{d Nak posi of direct image sheaf}
    Let $L$ be a holomorphic line bundle over a \kah manifold $X$ equipped with a smooth (semi)-positive Hermitian metric $h$ and $f:X\to Y$ be a proper holomorphic submersion between two complex manifolds.
    For the Kodaira-Spencer map $\rho_t:T_{Y,t}^{1,0}\to H^{0,1}(X_t,T^{1,0}_{X_t})$, if Kodaira-Spencer forms representing classes $\rho_t(\partial/\partial t_j)$ can be taken to be zero, 
    then the smooth canonical Hermitian metric $H$ on $f_*(K_{X/Y}\otimes L)$ has dual Nakano semi-positivity.
\end{theorem}

Examples of this theorem are projections maps from the direct product of manifolds (see Corollary \ref{d Nak posi for X*Y to Y}) 
and projectivized bundle $\pi:\mathbb{P}(V)\to Y$ for an ample vector bundle $V\to Y$ when $\mathrm{det}\,V$ has a metric satisfying certain condition (see Theorem \ref{d Nak posi for d.i.sheaf in subsection}).

Second, we consider the case where the metric $h$ on $L$ with semi-positivity is singular, i.e. $h$ is pseudo-effective. 
In this case, twisting the multiplier ideal sheaf $\mathscr{I}(h)$ further to the sheaf $\omega_{X/Y}\otimes L$, 
we study the positivity of the direct image sheaf $\mathcal{E}:=f_*(\omega_{X/Y}\otimes L\otimes\mathscr{I}(h))$, 
where $f:X\to Y$ is a projective and surjective holomorphic mapping between two connected complex manifolds.
It is known that the torsion-free coherent sheaf $\mathcal{E}=f_*(\omega_{X/Y}\otimes L\otimes\mathscr{I}(h))$ has a singular canonical Hermitian metric $H$ induced by $h$, 
and this metric satisfies the minimal extension property and is Griffiths semi-positive (see Theorem\,\ref{HPS18, Theorem21.1}, \cite{BP08}, \cite{PT18}, \cite{HPS18}).

We show that this singular canonical Hermitian metric $H$ on $\mathcal{E}$ has a locally $L^2$-type Nakano semi-positivity.
Let $Y(\mathcal{E})\subseteq Y$ denote the maximal open subset where $\mathcal{E}$ is locally free, then $Z_{\mathcal{E}}:=Y\setminus Y(\mathcal{E})$ is a closed analytic subset of codimension $\geq2$.
Here, we define (see Definition\,\ref{def of ext L2 subsheaf}) the set $\Sigma_H$ on $Y$ related to the unbounded-ness of $H$ by 
\begin{align*}
    \Sigma_H:=\{t\in Y\mid \mathcal{E}_t\subsetneq H^0(X_t,K_{X_t}\otimes L|_{X_t})\}.
\end{align*}
Using the set $\Sigma_H$, we have the following.

\begin{theorem}\label{H of full loc L2 Nak semi-posi}
    If $X$ is projective and there exists an analytic set $A$ such that $\Sigma_H\subseteq A$ then $H$ is full locally $L^2$-type Nakano semi-positive on $Y(\mathcal{E})$ as in Definition \ref{Def loc L2 Nak semi posi on coh sheaf}.
\end{theorem}

The restriction of $\mathcal{E}$ to $Y(\mathcal{E})$ is holomorphic vector bundle and the $L^2$-subsheaf of this vector bundle with respect to $H$ is denoted by $\mathscr{E}(H)\subseteq\mathcal{E}|_{Y(\mathcal{E})}$ over $Y(\mathcal{E})$ which analogous to multiplier ideal sheaves.
For a natural inclusion $j:Y(\mathcal{E})=Y\setminus Z_{\mathcal{E}}\hookrightarrow Y$, we define the natural extended $L^2$-subsheaf with respect to $H$ over $Y$ by $\mathscr{E}_Y(H):=j_*\mathscr{E}(H)$ as in Definition \ref{def of ext L2 subsheaf}.

\begin{theorem}
    If $X$ is projective and there exists an analytic set $A$ such that $\Sigma_H\subseteq A$ then the natural extended $L^2$-subsheaf $\mathscr{E}_Y(H)$ over $Y$ is coherent.
\end{theorem}

Finally, we consider the relationship between the minimal extension property and Nakano semi-positivity 
and show that if a torsion-free coherent sheaf has a metric satisfying the minimal extension property, this sheaf does not necessarily have a Nakano semi-positive metric.
As a concrete example, we show that the quotient holomorphic vector bundle $(\mathbb{P}^n\times\mathbb{C}^{n+1})/\mathcal{O}_{\mathbb{P}^n}(-1)$ over $\mathbb{P}^n$ does not have a Nakano semi-positive metric and has a metric satisfying the minimal extension property.

\section{Positivity of smooth Hermitian metrics and $L^2$-estimates}

In this section, we define various positivity for holomorphic vector bundles and investigate its equivalence condition.

Let $X$ be a complex manifold of complex dimension $n$ equipped with a Hermitian metric $\omega$ and $(E,h)$ be a holomorphic Hermitian vector bundle of rank $r$ over $X$.
Let $(U,(z_1,\ldots,z_n))$ be local coordinates and $D=D'^h+\overline{\partial}$ be the Chern connection of $(E,h)$. 
The Chern curvature tensor $\Theta_{E,h}=D^2=[D'^h,\overline{\partial}]$ is a $(1,1)$-form and is written as
\begin{align*}
    \Theta_{E,h}=\sum \Theta^h_{jk}dz_j\wedge d\overline{z}_k,
\end{align*}
where the coefficients $\Theta^h_{jk}=[D'^h_{z_j},\overline{\partial}_{z_k}]$ are defined operators on $U$ and $\overline{\partial}_{z_j}=\partial/\partial\overline{z}_j$.

The smooth Hermitian metric $h$ on $E$ is said to be Griffiths (semi)-positive if for any section $u$ of $E$ and any vector $v\in\mathbb{C}^n$ we have 
\begin{align*}
    \sum_{1\leq j,k\leq n}(\Theta^h_{jk}u,u)_hv_j\overline{v}_k>0 \quad (\geq0).
\end{align*}
And $h$ is said to be Nakano (semi)-positive if for any sections $u_j$ of $E$ we have 
\begin{align*}
    \sum_{1\leq j,k\leq n}(\Theta^h_{jk}u_j,u_k)_h>0 \quad (\geq0).
\end{align*}

There is a natural antilinear isometry between $E^*$ and $E$, which we will denote by $J$.
Denote the pairing between $E^*$ and $E$ by $\langle\cdot,\cdot\rangle$.
For any local section $u$ of $E$ and any local section $\xi$ of $E^*$, we have 
\begin{align*}
    \langle\xi,u\rangle=(u,J\xi)_h.
\end{align*}

Under the natural holomorphic structure on $E^*$, we get
\begin{align*}
    \overline{\partial}_{z_j}\xi=J^{-1}D'^h_{z_j}J\xi,
\end{align*}
and the Chern connection on $E^*$ is given by 
\begin{align*}
    D'^{h^*}_{z_j}\xi=J^{-1}\overline{\partial}_{z_j}J\xi.
\end{align*}
Then we have that 
\begin{align*}
    \overline{\partial}_{z_j}\langle\xi,u\rangle&=\langle\overline{\partial}_{z_j}\xi,u\rangle+\langle\xi,\overline{\partial}_{z_j}u\rangle,\\
    \partial_{z_j}\langle\xi,u\rangle&=\langle D'^{h^*}_{z_j}\xi,u\rangle+\langle\xi,D'^h_{z_j}u\rangle,\\
    0=[\partial_{z_j},\overline{\partial}_{z_k}]\langle\xi,u\rangle&=\langle\Theta^{h^*}_{jk}\xi,u\rangle+\langle\xi,\Theta^h_{jk}u\rangle.
\end{align*}

Therefore for any local sections $\xi_j\in C^\infty(E^*)$ and $u_j\in C^\infty(E)$ such that $u_j=J\xi_j$, we have that
\begin{align*}
    \sum(\Theta^{h^*}_{jk}\xi_j,\xi_k)_{h^*}=-\sum(\Theta^h_{jk}u_k,u_j)_h,
\end{align*}
and for any local sections $u,v\in C^\infty(E)$, we have that 
\begin{align*}
    \overline{\partial}_{z_k}\partial_{z_j}(u,v)_h
    =(D'^h_{z_j}u,D'^h_{z_k}v)_h+(\overline{\partial}_{z_k}D'^h_{z_j}u,v)_h
    +(u,D'^h_{z_k}\overline{\partial}_{z_j}v)_h+(\overline{\partial}_{z_k}u,\overline{\partial}_{z_j}v)_h.
\end{align*}

If $u$ is holomorphic, then $-\overline{\partial}_{z_k}D'^h_{z_j}u=\Theta^h_{jk}u$. 
Thus for any local sections $u,v\in\mathcal{O}(E)_x$, we get 
    \begin{align*}
        \frac{\partial^2}{\partial z_j\partial\overline{z}_k}(u,v)_h=(D'^h_{z_j}u,D'^h_{z_k}v)_h-(\Theta^h_{jk}u,v)_h.
    \end{align*}
In particular, if $u,v\in\mathcal{O}(E)_x$ satisfying $D'^hu=D'^hv=0$ at $x$ then 
\begin{align*}
    \partial_{z_j}\overline{\partial}_{z_k}(u,v)_h=-(\Theta^h_{jk}u,v)_h \quad \mathrm{at}\,\,x.
\end{align*}

Let $u=(u_1,\cdots,u_n)$ be an $n$-tuple of local holomorphic sections of $E$, i.e. $u_j\in\mathcal{O}(E)$. We define $T^h_u$, an $(n-1,n-1)$-form through 
\begin{align*}
    T^h_u:=\sum(u_j,u_k)_h\widehat{dz_j\wedge d\overline{z}_k}
\end{align*}
where $(z_1,\cdots,z_n)$ are local coordinates on $X$, and $\widehat{dz_j\wedge d\overline{z}_k}$ denotes the wedge product of all $dz_l$ and $d\overline{z}_l$ expect $dz_j$ and $d\overline{z}_k$, 
multiplied by a constant of absolute value $1$, i.e. $idz_j\wedge d\overline{z}_k\wedge \widehat{dz_j\wedge d\overline{z}_k}=dV_{\mathbb{C}^n}$.
Hence, if $D'^hu_j=0$ at $x$ then we get 
\begin{align*}
    \idd T^h_u=-\sum(\Theta^h_{jk}u_j,u_k)_hdV_{\mathbb{C}^n},
\end{align*}
at $x$ by the equation 
\begin{align*}
    \idd T^h_u&=\sum(D'^h_{z_j}u_j,D'^h_{z_k}u_k)_hdV_{\mathbb{C}^n}-\sum(\Theta^h_{jk}u_j,u_k)_hdV_{\mathbb{C}^n}\\
    &=||\sum D'^h_{z_j}u_j||^2_h-\sum(\Theta^h_{jk}u_j,u_k)_hdV_{\mathbb{C}^n}.
\end{align*}

\begin{proposition}\label{smooth case Nak semi and T}$(\mathrm{cf.}$\,\cite{Ber09},\,\cite{Rau15}$)$
    We have that 
    \begin{itemize}
        \item $h$ is Nakano semi-positive if and only if for any $x\in X$ and any $u_j\in\mathcal{O}(E)_x$ such that $D'^hu_j=0$ at $x$, the $(n-1,n-1)$-form $-T^h_u$ is plurisubharmonic at $x$, i.e. $-\idd T^h_u\geq0$.
        \item $h$ is Nakano semi-negative if and only if for any $x\in X$ and any $u_j\in\mathcal{O}(E)_x$, the $(n-1,n-1)$-form $T^h_u$ is plurisubharmonic at $x$, i.e. $\idd T^h_u\geq0$.
    \end{itemize}
\end{proposition}

We introduce another notion about Nakano-type positivity.

\begin{definition}$(\mathrm{cf.~}$\cite{Siu82},\,\cite{LSY13})   
    Let $X$ be a complex manifold of complex dimension $n$ and $(E,h)$ be a holomorphic Hermitian vector bundle of rank $r$ over $X$.
    $(E,h)$ is said to be $\it{dual~Nakano ~positive}$ (resp. $\it{dual~Nakano ~semi}$-$\it{positive}$) if $(E^*,h^*)$ is Nakano negative (resp. Nakano semi-negative).
\end{definition}

Let $\xi_j\in C^\infty(E^*)$ and $u_j\in C^\infty(E)$ be $r$-tuples of smooth sections $E^*$ such that $u_j=J\xi_j$.
If $h$ is dual Nakano semi-positive then 
\begin{align*}
    0\geq\sum(\Theta^{h^*}_{jk}\xi_j,\xi_k)_{h^*}=-\sum(\Theta^h_{jk}u_k,u_j)_h,
\end{align*}
i.e. $\sum(\Theta^h_{jk}u_k,u_j)_h\geq0$.
Enough to consider at each point, for any $x\in X$ and any $u_j\in C^\infty(E)_x$ if $\sum(\Theta^h_{jk}u_k,u_j)_h\geq0$ at $x$ then $h$ is dual Nakano semi-positive.
Hence, we have that $h$ is dual Nakano semi-positive if and only if $\sum(\Theta^h_{jk}u_k,u_j)_h\geq0$ at any points $x$, for any $u_j\in C^\infty(E)_x$.

\begin{definition}\label{def of T for dual Nakano}
    Let $u=(u_1,\cdots,u_n)$ be an $n$-tuple of local holomorphic sections of $E$, i.e. $u_j\in\mathcal{O}(E)$. We define $\widetilde{T}^h_u$, an $(n-1,n-1)$-form through 
    \begin{align*}
        \widetilde{T}^h_u:=\sum(u_k,u_j)_h\widehat{dz_j\wedge d\overline{z}_k}
    \end{align*}
    where $(z_1,\cdots,z_n)$ are local coordinates on $X$.
\end{definition}

\begin{proposition}\label{d Nak posi if and only if T condition}
    $h$ is dual Nakano semi-positive if and only if for any $x\in X$ and any $u_j\in\mathcal{O}(E)_x$ such that $D'^hu_j=0$ at $x$, 
    the $(n-1,n-1)$-form $-\widetilde{T}^h_u$ is plurisubharmonic at $x$, i.e. $-\idd \widetilde{T}^h_u\geq0$.
\end{proposition}

\begin{proof}
This yields the following calculation,
\begin{align*}
    0\geq\idd \widetilde{T}^h_u=-\sum(\Theta^h_{jk}u_k,u_j)_hdV_{\mathbb{C}^n}=\sum(\Theta^{h^*}_{jk}\xi_j,\xi_k)_{h^*}dV_{\mathbb{C}^n},
\end{align*}
where $\xi_j:=J^{-1}u_j\in\mathcal{E}(E^*)_x$.
\end{proof}

By using this proposition, we can examine dual Nakano semi-positivity of $h$ without using the dual metric $h^*$.
Finally, we introduce the H\"ormander's $L^2$-existence theorem. 

\begin{theorem}\label{Hormander L2-estimate}$(\mathrm{cf.~[Dem}$-$\mathrm{book,~ChapterVIII,~Theorem~6.1]})$
    Let $(X,\widehat{\omega})$ be a complete \kah manifold, $\omega$ be another \kah metric which is not necessarily complete and $(E,h)$ be a holomorphic vector bundle which satisfies $A_{h,\omega}:=[i\Theta_{E,h},\Lambda_\omega]\geq0$ on $\Lambda^{n,q}T^*_X\otimes E$.
    Then for any $\overline{\partial}$-closed $f\in L^2_{n,q}(X,E,h,\omega)$ there exists $u\in L^2_{n,q-1}(X,E,h,\omega)$ satisfies $\overline{\partial}u=f$ and 
    \begin{align*}
        \int_X|u|^2_{h,\omega}dV_{\omega}\leq\int_X\langle A^{-1}_{h,\omega}f,f\rangle_{h,\omega}dV_{\omega},
    \end{align*}
    where we assume that the right-hand side is finite.
\end{theorem}

\begin{lemma}\label{DNWZ22, Lemma4.7}$\mathrm{(cf.\,[DNWZ22,\,Lemma\,4.7]})$
    Let $U\subset\mathbb{C}^n$ be a domain, $\omega_1,\omega_2$ be any two Hermitian forms on $U$, and $E=U\times\mathbb{C}^r$ be trivial vector bundle on $U$ with a Hermitian metric.
    Let $\Theta\in C^0(X,\Lambda^{1,1}T^*_X\otimes\mathrm{End}(E))$ such that $\Theta^*=-\Theta$. Then 
    \begin{align*}
        \mathrm{Im}[i\Theta,\Lambda_{\omega_1}]=\mathrm{Im}[i\Theta,\Lambda_{\omega_2}],
    \end{align*}
    and for any $E$-valued $(n,1)$-form $u\in\mathrm{Im}[i\Theta,\Lambda_{\omega_1}]$,
    \begin{align*}
        \langle[i\Theta,\Lambda_{\omega_1}]^{-1}u,u\rangle_{\omega_1}dV_{\omega_1}=\langle[i\Theta,\Lambda_{\omega_2}]^{-1}u,u\rangle_{\omega_2}dV_{\omega_2}.
    \end{align*}
\end{lemma}

\section{Dual Nakano positivity of direct image sheaves}

\subsection{Smooth canonical Hermitian metric of direct image sheaves}

Let $X$ be a \kah manifold of dimension $m+n$ and $Y$ be a complex manifold of dimension $m$.
We consider a proper holomorphic submersion $f:X\to Y$. 
The \textit{relative canonical bundle} $K_{X/Y}$ corresponding to the map $f$ is 
\begin{align*}
    K_{X/Y}=K_X\otimes f^*K_Y^{-1}.
\end{align*}
When restricted to a generic fiber $X_t$ of $t$, we get $K_{X/Y}|_{X_t}\cong K_{X_t}$.

Let $L$ be a holomorphic Hermitian line bundle over $X$ equipped with a smooth semi-positive Hermitian metric $h$, i.e. $i\Theta_{L,h}\geq0$.
In this subsection, we discuss the complex structure of the direct image sheaf $f_*(K_{X/Y}\otimes L)$ on $Y$ and the smooth canonical Hermitian metric $H$ of this sheaf induced by $h$ (cf. \cite{Ber09}).
Fixed a point $t\in Y$, any section $u\in H^0(X_t,K_{X_t}\otimes L|_{X_t})$ extends in the sense that there is a holomorphic section 
\begin{align*}
    U\in H^0(f^{-1}(\Omega),K_X\otimes L|_{f^{-1}(\Omega)})\cong H^0(\Omega,K_Y\otimes f_*(K_{X/Y}\otimes L))
\end{align*}
such that $U|_{X_t}=u\wedge dt$ for some neighborhood $\Omega$ of $t$ from the Ohsawa-Takegoshi $L^2$-extension theorem (cf. \cite{OT87}) and K\"ahler-ness of $X$.
Here, we abusively denote by $dt$ the inverse image of a local generator $dt_1\wedge\cdots\wedge dt_m$ of $K_Y$.
In [Ber09], it was claimed that the total space
\begin{align*}
    F:=\bigcup_{t\in Y}H^0(X_t,K_{X_t}\otimes L|_{X_t})
\end{align*} 
has a natural structure of holomorphic vector bundle of rank $r:=h^0(X_t,K_{X_t}\otimes L_{X_t})$ over $Y$ and coincides with the direct image $f_*(K_{X/Y}\otimes L)$.
Therefore, the space of local \textit{smooth} sections of $F|_\Omega$ are simply the sections of the bundle $K_{X/Y}\otimes L|_{f^{-1}(\Omega)}$ whose restriction to each fiber of $f$ is holomorphic.

The vector bundle $F=f_*(K_{X/Y}\otimes L)$ admits a natural \textit{complex structure} as follows.
Let $u$ be a local section of $E$ then $u$ is holomorphic if 
\begin{align*}
    \overline{\partial}u\wedge dt=0.
\end{align*}
This is equivalent to saying that the section $u\wedge dt$ of $K_X\otimes L$ is holomorphic.

Note that $u$ is holomorphic, i.e. $\overline{\partial}u\wedge dt=0$, which means that $\overline{\partial}u$ can be written 
\begin{align*}
    \overline{\partial}u=\sum \eta^j\wedge dt_j,
\end{align*}
with $\eta^j$ smooth forms of bidegree $(n-1,1)$. 
Here, the following relationship is known (see \cite{Ber09}) between $\eta^j$ and the Kodaira-Spencer map $\rho_t:T^{1,0}_{Y,t}\to H^{0,1}(X_t,T^{1,0}_{X_t})$:
\begin{align*}
    \eta^j=\theta_j \rfloor u,
\end{align*}
on each fiber where the classes $\rho_t(\partial/\partial t_j)$ can be represented by Kodaira-Spencer forms $\theta_j$, i.e. $\{\theta_j\}\in\rho_t(\partial/\partial t_j)$. 

The smooth Hermitian metric $h$ of $L$ induces a smooth \textit{canonical Hermitian metric} $H$ of $F$ as follows.
Let $u,v$ be two local sections of $F$. We denote by $(u_t)$ the family of $L$-twisted holomorphic $(n,0)$-forms on fibers $K_{X_t}$ induced by $u$.
The restriction of $u_t$ to $X_t$ is unique and denoted simply as $u$.
Then the canonical Hermitian metric $H$ induced from $h$ is defined by
\begin{align*}
    (u,v)_H(t):=\int_{X_t}c_nu_t\wedge\overline{v}_te^{-\varphi}=\int_{X_t}c_nu\wedge\overline{v}e^{-\varphi},
\end{align*}
where $h=e^{-\varphi}$ on locally and $c_n=i^{n^2}$. This metric is smooth by Ehresmann's fibration theorem and compact-ness of each fiber.
And this inner product of $H$ is a function of $t$ and it will be convenient to write this function as 
\begin{align*}
    (u,v)_H=f_*(c_nu\wedge\overline{v}e^{-\varphi}),
\end{align*}
where $u$ and $v$ are forms on $X$ that represent the sections.
Here $f_*$ denotes the direct image of form defined by 
\begin{align*}
    \int_Yf_*(\alpha)\wedge\beta=\int_X\alpha\wedge f^*(\beta),
\end{align*}
if $\alpha$ is a form on $X$ and $\beta$ is a form on $Y$.

\subsection{Berndtsson calculation and Nakano positivity}

Let $(t_1,\cdots,t_m)$ be a local coordinate whose center is fixed point $y\in Y$.
Let $u_j$ be an $m$-tuple of local holomorphic sections to $F$ that satisfy $D'^Hu_j=0$ at $y$, i.e. $t=0$.
Represent the $u_j$ by smooth forms on $X$ and put 
\begin{align*}
    \hat{u}:=\sum u_j\wedge\widehat{dt_j}
\end{align*}
then we get 
\begin{align*}
    T_u^H=c_Nf_*(\hat{u}\wedge\overline{\hat{u}}e^{-\varphi}),
\end{align*}
where $N=n+m-1$ and $\widehat{dt_j}$ is the wedge product of all differentials $dt_k$ except $dt_j$ such that $dt_j\wedge\widehat{dt_j}=dt=dt_1\wedge\cdots\wedge dt_m$.

Using the following proposition, Berndtsson computed $\idd T^H_u$ at fixed points.

\begin{proposition}\label{Ber09, Prop4.2}$(\mathrm{cf.~[Ber09,\,Proposition\,4.2]})$
    Let $u$ be a section of $F$ over an open set $U$ containing the origin such that $\overline{\partial}u=0$ in $U$, i.e. holomorphic, and $D'^Hu=0$ at $t=0$.
    Then $u$ can be represented by a smooth $(n,0)$-form, still denoted $u$ such that 
    \begin{align*}
        \overline{\partial}u=\sum\eta^k\wedge dt_k,
    \end{align*}
    where $\eta^k$ is $\mathit{primitive}$ on $X_0$, i.e. satisfies $\eta^k\wedge\omega=0$ on $X_0$, and furthermore
    \begin{align*}
        \partial^\varphi u\wedge\widehat{dt_j}=0,
    \end{align*}
    at $t=0$ for all $j$. Here, $\partial^\varphi\cdot=e^\varphi\partial(e^{-\varphi}\cdot)$.
\end{proposition}

Let $u_j\in\mathcal{O}(F)$ such that $D'^Hu_j=0$ at $t=0$, then we have that 
    \begin{align*}
        \idd T^H_u=-c_Nf_*(\hat{u}\wedge\overline{\hat{u}}\wedge\idd\varphi e^{-\varphi}) -\Bigl(\int_{X_0}|\eta|^2e^{-\varphi}dV_z\Bigr)dV_t,
    \end{align*}
at $t=0$, where $\overline{\partial}u_j=\sum\eta^k_j\wedge dt_k$ and $\eta=\sum\eta^j_j$.

From this calculation and Proposition \ref{smooth case Nak semi and T}, the following theorem is obtained.

\begin{theorem}\label{Ber09, Thm1.2}$(\mathrm{cf.~[Ber09,\,Theorem\,1.2]})$
    If $L$ has a smooth (semi)-positive Hermitian metric then the smooth canonical Hermitian metric $H$ on $F=f_*(K_{X/Y}\otimes L)$ is Nakano (semi)-positive.
\end{theorem}

\subsection{Calculation of $\widetilde{T}^H_u$ for the canonical Hermitian metric on $f_*(K_{X/Y}\otimes L)$} 

Represent the $u_j$ by smooth forms on $X$ and put 
\begin{align*}
    \tilde{u}:=\sum \overline{u}_j\wedge\widehat{dt_j}
\end{align*}
then we have the equality
\begin{align*}
    \widetilde{T}_u^H=(-1)^nc_Nf_*(\tilde{u}\wedge\overline{\tilde{u}}e^{-\varphi}),
\end{align*}
where using $ic_N=(-1)^N(-1)^{nm}c_nc_m$ and $ic_m(-1)^m\widehat{dt_j}\wedge\widehat{d\overline{t}_k}=\widehat{dt_j\wedge d\overline{t}_k}$.


In this subsection, we show the following proposition.

\begin{proposition}\label{calculation of dual Nakano positivity}
    Let $u_j\in\mathcal{O}(F)$ such that $D'^Hu_j=0$ at $t=0$, then we have that 
    \begin{align*}
        \idd \widetilde{T}^H_u
        =-c_Nf_*(\hat{v}\wedge\overline{\hat{v}}\wedge\idd\varphi e^{-\varphi})+c_n\Bigl(\int_{X_0}\sum\eta^j_k\wedge\overline{\eta}^k_j e^{-\varphi}\Bigr)dV_t
    \end{align*}
    at $t=0$, where $u_j=U_jdz,\, \hat{v}=\sum \overline{U}_j\wedge dz \wedge \widehat{dt_j}$ and $\overline{\partial}u_j=\sum\eta^k_j\wedge dt_k$.
    Here $c_Nf_*(\hat{v}\wedge\overline{\hat{v}}\wedge\idd\varphi e^{-\varphi})\geq0$ if $\varphi$ is plurisubharmonic.

    In particular, if $\eta^j_k$ is primitive on $X_0$ then we can write $\eta^j_k=\sum \eta_{jkl}\widehat{dz_l}\wedge d\overline{z}_l$ and get 
    \begin{align*}
        c_n\Bigl(\int_{X_0}\sum\eta^j_k\wedge\overline{\eta}^k_j e^{-\varphi}\Bigr)dV_t=-\Bigl(\int_{X_0}\sum\eta_{jkl}\overline{\eta}_{kjl} e^{-\varphi}dV_z\Bigr)dV_t.
    \end{align*}
\end{proposition}

\begin{proof}
    By $\partial^\varphi\cdot=e^\varphi\partial(e^{-\varphi}\cdot)$, we get 
    \begin{align*}
        (-1)^n\partial \widetilde{T}^H_u
        =c_Nf_*(\partial\tilde{u}\wedge\overline{\tilde{u}}e^{-\varphi})+(-1)^Nc_Nf_*(\tilde{u}\wedge\partial^\varphi\overline{\tilde{u}}e^{-\varphi}).
    \end{align*}

    From the equation
    \begin{align*}
        \partial\tilde{u}=\sum \partial\overline{u}_j\wedge\widehat{dt_j}=\sum \overline{\overline{\partial}u_j}\wedge\widehat{dt_j}=\sum \overline{\eta}^l_j\wedge d\overline{t}_l\wedge\widehat{dt_j},
    \end{align*}
    the form 
    \begin{align*}
        \partial\tilde{u}\wedge\overline{\tilde{u}}&=\sum \overline{\eta}^l_j\wedge d\overline{t}_l\wedge\widehat{dt_j}\wedge u_k\wedge\widehat{d\overline{t}_k}\\
        &=(-1)^N\sum \overline{\eta}^l_j\wedge \widehat{dt_j}\wedge u_k\wedge d\overline{t}_l\wedge\widehat{d\overline{t}_k}\\
        &=(-1)^N\sum \overline{\eta}^k_j\wedge \widehat{dt_j}\wedge u_k\wedge d\overline{t}
    \end{align*}
    contains a factor $d\overline{t}$. On the other hand, the push forward of an $(n+m,n+m-1)$-form is of bidegree $(m,m-1)$. 
    Then we get 
    \begin{align*}
        f_*(\partial\tilde{u}\wedge\overline{\tilde{u}}e^{-\varphi})=0.
    \end{align*}

    Thus we have that 
    \begin{align*}
        (-1)^n\overline{\partial}\partial\widetilde{T}^H_u
        =(-1)^Nc_Nf_*(\overline{\partial}^\varphi\tilde{u}\wedge\partial^\varphi\overline{\tilde{u}}e^{-\varphi})+c_Nf_*(\tilde{u}\wedge\overline{\partial}\partial^\varphi\overline{\tilde{u}}e^{-\varphi}).
    \end{align*}
    By the equation $\overline{\partial}\partial^\varphi+\partial^\varphi\overline{\partial}=\partial\overline{\partial}\varphi$, we get 
    \begin{align*}
        c_Nf_*(\tilde{u}\wedge\overline{\partial}\partial^\varphi\overline{\tilde{u}}e^{-\varphi})
        =c_Nf_*(\tilde{u}\wedge\overline{\tilde{u}}\wedge\partial\overline{\partial}\varphi e^{-\varphi})-c_Nf_*(\tilde{u}\wedge\partial^\varphi\overline{\partial\tilde{u}}e^{-\varphi}),
    \end{align*}
    and by the vanishing $f_*(\tilde{u}\wedge\overline{\partial\tilde{u}}e^{-\varphi})=0$, we get 
    \begin{align*}
        0=\partial f_*(\tilde{u}\wedge\overline{\partial\tilde{u}}e^{-\varphi})
        =f_*(\partial\tilde{u}\wedge\overline{\partial\tilde{u}}e^{-\varphi})+(-1)^Nf_*(\tilde{u}\wedge\partial^\varphi\overline{\partial\tilde{u}}e^{-\varphi}).
    \end{align*}

    Hence, we have that 
    \begin{align*}
        (-1)^n\overline{\partial}\partial\widetilde{T}^H_u
        &=(-1)^Nc_Nf_*(\overline{\partial}^\varphi\tilde{u}\wedge\partial^\varphi\overline{\tilde{u}}e^{-\varphi})\\
        &\qquad+c_Nf_*(\tilde{u}\wedge\overline{\tilde{u}}\wedge\partial\overline{\partial}\varphi e^{-\varphi})+(-1)^Nc_Nf_*(\partial\tilde{u}\wedge\overline{\partial\tilde{u}}e^{-\varphi}).
    \end{align*}

    Note that with the choice of representatives of our sections $u_j$ furnished by Proposition \ref{Ber09, Prop4.2}, we have that $\overline{\partial}^\varphi\tilde{u}\wedge\partial^\varphi\overline{\tilde{u}}=0$ at $t=0$. 
    In fact, $\overline{\partial}^\varphi\tilde{u}=\sum\overline{\partial^\varphi u_j}\wedge\widehat{dt_j}$ and 
    \begin{align*}
        \overline{\partial}^\varphi\tilde{u}\wedge\partial^\varphi\overline{\tilde{u}}=\sum\overline{\partial^\varphi u_j}\wedge\widehat{dt_j}\wedge \partial^\varphi u_k\wedge\widehat{d\overline{t}_k}=0
    \end{align*}
    at $t=0$, where $\partial^\varphi u_j\wedge\widehat{dt_k}=0$ at $t=0$ for all $k$.

    \begin{lemma}
        We have that 
        \begin{align*}
            (-1)^Nc_Nf_*(\partial\tilde{u}\wedge\overline{\partial\tilde{u}}e^{-\varphi})=ic_n(-1)^n\Bigl(\int_{X_0}\sum \eta^j_k\wedge\overline{\eta}^k_je^{-\varphi}\Bigr)dV_t
        \end{align*}
        at $t=0$.
        In particular, if $\eta^j_k$ is primitive on $X_0$, i.e. $\eta^j_k\wedge \omega=0$ on $X_0$, then we can write $\eta^j_k=\sum \eta_{jkl}\widehat{dz_l}\wedge d\overline{z}_l$ and this integral value is
        \begin{align*}
            -i(-1)^n\Bigl(\int_{X_0}\sum \eta_{jkl}\overline{\eta}_{kjl}e^{-\varphi}dV_z\Bigr) dV_t.
        \end{align*}
    \end{lemma}

    \begin{proof}
        Here $\partial\tilde{u}=\sum \overline{\eta}^l_j\wedge d\overline{t}_l\wedge\widehat{dt_j}$, then
        \begin{align*}
            \partial\tilde{u}\wedge\overline{\partial\tilde{u}}
            =-(-1)^{nm}(-1)^{n^2}\sum \eta^j_k\wedge \overline{\eta}^k_j\wedge dt\wedge d\overline{t}.
        \end{align*}
        Therefore we get 
        \begin{align*}
            (-1)^Nc_N\partial\tilde{u}\wedge\overline{\partial\tilde{u}}&=-(-1)^Nc_N(-1)^{nm}(-1)^{n^2}\sum \eta^j_k\wedge \overline{\eta}^k_j\wedge dt\wedge d\overline{t}\\
            &=ic_nc_m(-1)^n\sum \eta^j_k\wedge \overline{\eta}^k_j\wedge dt\wedge d\overline{t}\\
            &=ic_n(-1)^n\sum \eta^j_k\wedge \overline{\eta}^k_j\wedge dV_t,
        \end{align*}
        where $ic_N=(-1)^N(-1)^{nm}c_nc_m$.

        If $\eta^j_k$ is primitive, we can write $\eta^j_k=\sum \eta_{jkl}\widehat{dz_l}\wedge d\overline{z}_l$. Then we get 
        \begin{align*}
            \eta^j_k\wedge \overline{\eta}^k_j&=\sum \eta_{jkl}\widehat{dz_l}\wedge d\overline{z}_l \wedge \overline{\eta}_{kj\mu}\widehat{d\overline{z}_\mu}\wedge dz_\mu\\
            &=\sum \eta_{jkl}\overline{\eta}_{kjl}\widehat{dz_l}\wedge d\overline{z}_l \wedge \widehat{d\overline{z}_l}\wedge dz_l\\
            &=(-1)^{2n-1}\sum \eta_{jkl}\overline{\eta}_{kjl} dz_l\wedge\widehat{dz_l}\wedge d\overline{z}_l \wedge \widehat{d\overline{z}_l}\\
            &=-\sum \eta_{jkl}\overline{\eta}_{kjl} dz\wedge d\overline{z},
        \end{align*}
        where $c_ndz\wedge d\overline{z}=dV_z$.
    \end{proof}

    Hence, we have that 
    \begin{align*}
        i\partial\overline{\partial}\widetilde{T}^H_u&=-(-1)^nc_Nf_*(\tilde{u}\wedge\overline{\tilde{u}}\wedge\idd\varphi e^{-\varphi})+c_n\Bigl(\int_{X_0}\sum\eta^j_k\wedge\overline{\eta}^k_j e^{-\varphi}\Bigr)dV_t
    \end{align*}
    at $t=0$. Let $u_j=U_jdz$ and $\varphi_{jk}:=\partial_{t_j}\overline{\partial}_{t_k}\varphi$.
    Here if $\varphi$ is plurisubharmonic then $c_N\hat{u}\wedge\overline{\hat{u}}\wedge\idd\varphi=\sum \varphi_{jk}U_j\overline{U}_kdV_z \wedge dV_t\geq0$.
    By $\tilde{u}=\sum \overline{u}_j\wedge \widehat{dt_j}=\sum \overline{U}_j d\overline{z} \wedge \widehat{dt_j}$ and $d\overline{z}\wedge dz=(-1)^{n^2}dz\wedge d\overline{z}=(-1)^ndz\wedge d\overline{z}$, we have that 
    \begin{align*}
        c_N\tilde{u}\wedge\overline{\tilde{u}}\wedge\idd\varphi=(-1)^n\sum \varphi_{jk}\overline{U}_j U_kdV_z \wedge dV_t
        =(-1)^nc_N\hat{v}\wedge\overline{\hat{v}}\wedge\idd\varphi,
    \end{align*}
    where $\hat{v}=\sum \overline{U}_j\wedge dz \wedge \widehat{dt_j}$ and that 
    \begin{align*}
        (-1)^nc_Nf_*(\tilde{u}\wedge\overline{\tilde{u}}\wedge\idd\varphi e^{-\varphi})&=c_Nf_*(\hat{v}\wedge\overline{\hat{v}}\wedge\idd\varphi e^{-\varphi})\\
        &=f_*(\sum \varphi_{jk}\overline{U}_j U_k e^{-\varphi}dV_z \wedge dV_t)\\
        &=\Bigl(\int_{X_0}\sum \varphi_{jk}\overline{U}_j U_k e^{-\varphi}dV_z\Bigr)dV_t\\
        &\geq0,
    \end{align*}
    if $\varphi$ is plurisubharmonic.
\end{proof}

\subsection{Proof of Theorem \ref{d Nak posi of direct image sheaf} and projectivized bundles}

Let $V$ be a holomorphic vector bundle of finite rank $r$ over a compact complex manifold $Y$.
Let $\pi:\mathbb{P}(V)\to Y$ be a projectivized bundle whose fiber at $t\in Y$ is the projective space of lines in $V^*_t$, i.e. $\mathbb{P}(V^*_t)$.
For any point $t\in Y$, we get $\pi^{-1}(t)=\mathbb{P}(V^*_t)\cong\mathbb{P}^{r-1}$ then $\mathbb{P}(V)$ is a holomorphically locally trivial fibration.
This projectivized bundle carries the tautological line bundle $\mathcal{O}_{\mathbb{P}(V)}(1)$ over $\mathbb{P}(V)$ whose restriction to any fiber $\mathbb{P}(V^*_t)$ is identical to $\mathcal{O}_{\mathbb{P}^{r-1}}(1)$.

We shall apply Proposition \ref{calculation of dual Nakano positivity} to the line bundles $\mathcal{O}_{\mathbb{P}(V)}(k)\to\mathbb{P}(V)$ where $k\in\mathbb{Z}$.
Let $E(k)$ be the vector bundle whose fiber over a point $t\in Y$ is the space of global holomorphic sections of $K_{\mathbb{P}(V^*_t)}\otimes\mathcal{O}_{\mathbb{P}(V)}(k)$, i.e. 
\begin{align*}
    E(k)&:=\bigcup_{t\in Y}H^0(\mathbb{P}(V^*_t),K_{\mathbb{P}(V^*_t)}\otimes\mathcal{O}_{\mathbb{P}(V)}(k)|_{\mathbb{P}(V^*_t)})\\
    &=\pi_*(K_{\mathbb{P}(V)/Y}\otimes\mathcal{O}_{\mathbb{P}(V)}(k)),
\end{align*}
where $E(k)_t=H^0(\mathbb{P}(V^*_t),K_{\mathbb{P}(V^*_t)}\otimes\mathcal{O}_{\mathbb{P}(V)}(k)|_{\mathbb{P}(V^*_t)})\cong H^0(\mathbb{P}^{r-1},\mathcal{O}_{\mathbb{P}^{r-1}}(k-r))$.
If $k<r$ then each fiber $E(k)_t$ is zero. 
Berndtsson asserted the following fact
\begin{align*}
    E(r+m)=S^m(V)\otimes\mathrm{det}\,V,
\end{align*}
where $S^m(V)$ is the $m$-th symmetric power of $V$, and showed the following theorem using Theorem \ref{Ber09, Thm1.2}. 

\begin{theorem}\label{Ber09, Thm1.3}$(\mathrm{cf.~[Ber09,\,Theorem\,1.3]})$
    Let $V$ be a (finite rank) holomorphic vector bundle over a complex manifold. 
    If $\mathcal{O}_{\mathbb{P}(V)}(1)$ has a smooth (semi)-positive metric, then $V\otimes\mathrm{det}\,V$ has a smooth canonical Hermitian metric which is Nakano (semi)-positive.
\end{theorem}

Here, the vector bundle $V$ is called ample in the sense of Hartshorne (see \cite{Har66}) if the tautological line bundle $\mathcal{O}_{\mathbb{P}(V)}(1)$ is ample.
Replacing $\mathcal{O}_{\mathbb{P}(V)}(r+1)$ by $\mathcal{O}_{\mathbb{P}(V)}(r+m)$, we also get that $S^m(V)\otimes\mathrm{det}\,V$ is Nakano (semi)-positivity for any $m\in\mathbb{N}$.

It is well known Griffiths conjecture that an ample vector bundle is Griffiths positive, i.e. has a smooth Griffiths positive Hermitian metric.
From Demailly-Skoda's theorem (see \cite{DS80}) that if $V$ is Griffiths (semi)-positive then $V\otimes\mathrm{det}\,V$ is (dual) Nakano (semi)-positive, this theorem may be regarded as indirect evidence of Griffiths conjecture.
After that, it was shown that $S^m(V)\otimes\mathrm{det}\,V$ has Nakano-positive metric and dual Nakano-positive metric (see [LSY13,\,Corollary\,4.12]).
Griffiths conjecture is known when $Y$ is a compact curve (cf. \cite{Ume73}), and it was recently shown to hold under a certain condition for the $L^2$ metric (see \cite{Nau21}). 
Since the Kodaira-Spencer forms vanishes under certain condition in \cite{Nau21}, we obtain the following theorem for dual Nakano positivity of the canonical Hermitian metric which is a different metric in \cite{LSY13}. 

\begin{theorem}\label{d Nak posi for d.i.sheaf in subsection}
    Let $V$ be an ample holomorphic vector bundle of rank $r$ over a complex manifold $Y$. 
    If 
    the canonical isomorphism 
    \begin{align*}
        K^{-1}_{\mathbb{P}(V)/Y}\cong \mathcal{O}_{\mathbb{P}(V)}(r)\otimes\pi^*\mathrm{det}\,V^*
    \end{align*}
    becomes an isometry for an positive metric on $\mathcal{O}_{\mathbb{P}(V)}(1)$ and some Hermitian metric on $\mathrm{det}\,V$, 
    then for any $m\in\mathbb{N}$ and for a smooth (semi)-positive Hermitian metric $h$ on $\mathcal{O}_{\mathbb{P}(V)}(r+m)$, the smooth canonical Hermitian metric $H$ induced by $h$ on $S^m(V)\otimes\mathrm{det}\,V$ is dual Nakano (semi)-positive.
\end{theorem}

We prove this below. Let $(t_1,\ldots,t_m)$ and $(z_1,\ldots,z_n)$ be local coordinates on $Y$ and the fibers respectively.
By ampleness of $V$, there is a smooth positive Hermitian metric $h_{O(1)}$ on $\mathcal{O}_{\mathbb{P}(V)}(1)$.
We write locally for the curvature of the positively curved metric
\begin{align*}
    \omega_{\mathbb{P}(V)}:&=-\idd\log h_{O(1)}\\
    &=i\Bigl(g_{\alpha\overline{\beta}}dz_\alpha\wedge d\overline{z}_\beta
    +h^{O(1)}_{k\overline{\beta}}dt_k\wedge d\overline{z}_\beta
    +h^{O(1)}_{\alpha\overline{l}}dz_\alpha\wedge d\overline{z}_l
    +h^{O(1)}_{k\overline{l}}dz_k\wedge d\overline{z}_l\Bigr).
\end{align*}
Thus the \kah forms on each fibers are given by $\omega_t:=i\sum g_{\alpha\overline{\beta}}dz_\alpha\wedge d\overline{z}_\beta$ and the induced metric on $K^{-1}_{\mathbb{P}(V)/Y}$ can be written as $\mathrm{det}\,(g_{\alpha\overline{\beta}})$.
Here, this positive metric $h_{O(1)}$ induces the above canonical isomorphism.

According to \cite{Sch93}, we denote the horizontal lift of a tangent vector $\partial/\partial t_j$ on the base $Y$ by $v_j$. It is given by 
\begin{align*}
    v_j=\frac{\partial}{\partial t_j}+\sum a^\alpha_j\frac{\partial}{\partial z_\alpha} \quad \mathrm{and} \quad a^\alpha_j=-\sum g^{\overline{\beta}\alpha}h^{O(1)}_{j\overline{\beta}}.
\end{align*} 
For a fibration $\pi:\mathbb{P}(V)\to Y$, we obtain the Kodaira-Springer forms by 
\begin{align*}
    \theta_j:=\overline{\partial}(v_j)|_{X_t},
\end{align*}
where $\theta_j\in \rho_t(\partial/\partial t_j)$.

\begin{proposition}\label{Nau21 Prop1}$(\mathrm{cf.~[Nau21,\,Proposition\,1]})$
    Under the assumption of Theorem \ref{d Nak posi for d.i.sheaf in subsection}, the Kodaira-Spencer forms $\theta_j$ are harmonic, hence zero.
\end{proposition}

Since it is a projectivized bundle, we get $\mathscr{H}^{0,1}(\mathbb{P}(V^*_t),T^{1,0}_{\mathbb{P}(V^*_t)})\cong H^{0,1}(\mathbb{P}(V^*_t),T^{1,0}_{\mathbb{P}(V^*_t)})\cong H^{0,1}(\mathbb{P}^{r-1},T^{1,0}_{\mathbb{P}^{r-1}})=0$.
Then the value of Kodaira-Spencer map is zero.  
Here $\{\theta_j\}=\rho_t(\partial/\partial t_j)=0$. By the forms $\theta_j$ is harmonic, $\theta_j$ is zero as differential forms.

\vspace{3mm}

$\textit{Proof of Theorem \ref{d Nak posi for d.i.sheaf in subsection}}$.
    From the Kodaira-Spencer forms $\theta_j$ are zero and the definition of the complex structure in $E(r+m)$, for any local holomorphic section $u\in \mathcal{O}(E(r+m))_t$, the restriction of
    \begin{align*}
        \overline{\partial}u=\sum \eta^j\wedge dt_j
    \end{align*}
    to each fiber is zero. In fact the smooth $(n-1,1)$-forms $\eta^j$ equals $\theta_j\rfloor u$ in each fiber.
    In particular, we get $\eta^j=\theta_j\rfloor u=0$ in each fiber.

    By Proposition \ref{calculation of dual Nakano positivity}, for any local holomorphic section $u_j\in\mathcal{O}(E(r+m))$ such that $D'^Hu_j=0$ at $t=0$ we have that 
    \begin{align*}
        \idd\widetilde{T}^H_u=-c_N\pi_*(\hat{v}\wedge\overline{\hat{v}}\wedge\idd\varphi e^{-\varphi})
    \end{align*}
    at $t=0$, where $u_j=U_jdz,\, \hat{v}=\sum \overline{U}_j\wedge dz \wedge \widehat{dt_j}$ and $\varphi=-\log h$ on locally.
    Here $c_N\pi_*(\hat{v}\wedge\overline{\hat{v}}\wedge\idd\varphi e^{-\varphi})\geq0 \,\,\,(resp. \,>0)$ if $\varphi$ is (strictly) plurisubharmonic.

    Hence, this theorem follows from Proposition \ref{d Nak posi if and only if T condition}. \qed

\vspace{2mm}

Similar to this proof, Theorem \ref{d Nak posi of direct image sheaf} can be shown from Proposition \ref{calculation of dual Nakano positivity}, since if Kodaira-Spencer forms $\theta_j$ can be taken to be zero then $\eta^j$ vanishes where $\theta_j\in\rho_t(\partial/\partial t_j)$.
And the following corollary is obtained.

\begin{corollary}\label{d Nak posi for X*Y to Y}
    Let $X$ be a compact \kah manifold, $Y$ be a complex manifold and $L$ be a holomorphic vector bundle over $Z:=X\times Y$ equipped with a smooth semi-positive Hermitian metric $h$. Let $\pi:Z=X\times Y\to Y$ be a natural projection map.
    Then the smooth canonical Hermitian metric $H$ on $\pi_*(K_{Z/Y}\otimes L)$ has dual Nakano semi-positivity.
\end{corollary}

\section{Singular Hermitian metric and Positivity}

\subsection{Singular Hermitian metric on vector bundle and positivity}

For any holomorphic vector bundle $E$, we introduce the definition of singular Hermitian metrics $h$ on $E$ and 
the $L^2$-subsheaf $\mathscr{E}(h)$ of $\mathscr{O}(E)$ analogous to the multiplier ideal sheaf.

\begin{definition}$(\mathrm{cf.\,[BP08,\,Section\,3],\,[PT18,\,Definition,\,2.2.1]})$ 
    We say that $h$ is a $\it{singular~Hermitian~metric}$ on $E$ if $h$ is a measurable map from the base manifold $X$ to the space of non-negative Hermitian forms on the fibers satisfying $0<\mathrm{det}\,h<+\infty$ almost everywhere.
\end{definition}

\begin{definition}\label{def of L2 subsheaf}$(\mathrm{cf.\,[deC98,\,Definition\,2.3.1]})$ 
    Let $h$ be a singular Hermitian metric on $E$. We define the $L^2$-subsheaf $\mathscr{E}(h)$ of germs of local holomorphic sections of $E$ by
    \begin{align*}
        \mathscr{E}(h)_x:=\{s_x\in\mathscr{O}(E)_x\mid|s_x|^2_h~ \mathrm{is ~locally ~integrable ~around~} x\}.
    \end{align*}
\end{definition}

If $E$ is a holomorphic line bundle then we get $\mathscr{E}(h)=\mathscr{O}(E)\otimes\mathscr{I}(h)$.
Moreover, we define positivity and negativity such as Griffiths and dual Nakano. 

\begin{definition}\label{def Griffiths semi-posi sing}$(\mathrm{cf.\,[BP08,\,Definition\,3.1],\,[PT18,\,Definition~2.2.2]})$  
    We say that a singular Hermitian metric $h$ is 
    \begin{itemize}
        \item [(1)] $\textit{Griffiths semi-negative}$ if $|u|_h$ is plurisubharmonic for any local holomorphic section $u\in\mathscr{O}(E)$ of $E$.
        \item [(2)] $\textit{Griffiths semi-positive}$ if the dual metric $h^*$ on $E^*$ is Griffiths semi-negative.
    \end{itemize}
\end{definition}

\begin{definition}\label{def Nakano semi-negative as Raufi}$\mathrm{(cf.\,[Rau15,\,Section\,1]})$ 
    We say that a singular Hermitian metric $h$ on $E$ is $\it{Nakano}$ $\it{semi}$-$\it{negative}$ if the $(n-1,n-1)$-form $T^h_u$ is plurisubharmonic for any $n$-tuple of local holomorphic sections $u=(u_1,\cdots,u_n)$.
\end{definition}

\begin{definition}\label{def dual Nakano semi-posi sing}$\mathrm{(cf.\,[Wat22a,\,Definition\,4.5]})$ 
    We say that a singular Hermitian metric $h$ on $E$ is $\it{dual}$ $\it{Nakano}$ $\it{semi}$-$\it{positive}$ if the dual metric $h^*$ on $E^*$ is Nakano semi-negative.
\end{definition}

For singular Hermitian metrics, we cannot always define the curvature currents with measure coefficients (see \cite{Rau15}).
However, the above definitions can be defined by not using the curvature currents.
In general, the dual of a Nakano negative bundle is not Nakano positive then we cannot define Nakano semi-positivity as in the case of Griffiths,
but this definition of dual Nakano semi-positivity is natural.
The characterization of Nakano semi-positivity using $L^2$-estimate by the following definition is already known by Deng-Ning-Wang-Zhou's work (see \cite{DNWZ22}).

\begin{definition}\label{def optimal L2 estimate on vector bdl}$\mathrm{(cf.\,[DNWZ22,\,Definition\,1.1]})$
    Let $X$ be a complex manifold of dimension $n$ and $U$ be an open subset of $X$ with a \kah metric $\omega$ on $U$ which admits a positive Hermitian holomorphic line bundle. Let $(E,h)$ be a (singular) Hermitian vector bundle over $X$.
    We call $(E,h)$ satisfies \textit{the optimal} $L^2$-\textit{estimate} on $U$ if for any positive Hermitian holomorphic line bundle $(A,h_A)$ on $U$, for any $f\in\mathscr{D}^{n,1}(U,E\otimes A)$ with $\overline{\partial}f=0$ on $U$, 
    there is $L^2_{n,0}(U,E\otimes A)$ satisfying $\overline{\partial}u=f$ on $U$ and 
    \begin{align*}
        \int_U|u|^2_{h\otimes h_A,\omega}dV_\omega\leq\int_U\langle B^{-1}_{A,h_A}f,f\rangle_{h\otimes h_A,\omega}dV_\omega,
    \end{align*}
    provide that the right hand side is finite, where $B_{A,h_A}=[i\Theta_{A,h_A}\otimes\mathrm{id}_E,\Lambda_\omega]$.
\end{definition}

Here, $\mathscr{D}$ denotes the space of $C^\infty$ sections with compact support.
Modifying the optimal $L^2$-estimate condition, one definition of Nakano semi-positivity that establishes vanishing theorems was introduced in \cite{Ina22}.

\begin{definition}\label{def Nakano semi-posi sing by Inayama}$\mathrm{(cf.\,[Ina22,\,Definition\,1.1]})$ 
    Assume that $h$ is a Griffiths semi-positive singular Hermitian metric. We say that $h$ is 
    (\textit{globally}) \textit{Nakano semi-positive} on $X$ if for any Stein coordinate $S\subseteq X$ such that $E|_S$ is trivial, any \kah metric $\omega_S$ on $S$,
    any smooth strictly plurisubharmonic function $\psi$ on $S$, any positive integer $q\in\{1,\cdots,n\}$ and any $\overline{\partial}$-closed $f\in L^2_{n,q}(S,E,he^{-\psi},\omega_S)$ 
    there exists $u\in L^2_{n,q-1}(S,E,he^{-\psi},\omega_S)$ satisfying $\overline{\partial}u=f$ and 
    \begin{align*}
        \int_S|u|^2_{h,\omega_S}e^{-\psi}dV_{\omega_S}\leq\int_S\langle B^{-1}_{\psi,\omega_S}f,f\rangle_{h,\omega_S}e^{-\psi}dV_{\omega_S},
    \end{align*}
    where $B_{\psi,\omega_S}=[\idd\psi\otimes\mathrm{id}_E,\Lambda_{\omega_S}]$. Here we assume that the right-hand side is finite.
\end{definition}

It is already known that multiplier ideal sheaves are coherent in \cite{Nad89}.
After that, Hosono and Inayama proved that the $L^2$-subsheaf $\mathscr{E}(h)$ is coherent if $h$ is Nakano semi-positive in the singular sense as in Definition \ref{def Nakano semi-posi sing by Inayama} (or \ref{def optimal L2 estimate on vector bdl}) in \cite{HI20} and \cite{Ina22}.

\subsection{Singular Hermitian metrics on torsion-free sheaves and positivity}

Let $X$ be a complex manifold and $\mathscr{F}$ be a torsion-free coherent sheaf on $X$. 
Let $X(\mathscr{F})\subseteq X$ denote the maximal open subset where $\mathscr{F}$ is locally free, 
then $Z_\mathscr{F}:=X\setminus X(\mathscr{F})$ is a closed analytic subset of codimension $\geq2$. 
If $\mathscr{F}\ne0$, then the restriction of $\mathscr{F}$ to the open subset $X(\mathscr{F})$ is a holomorphic vector bundle $F$ of some rank $r\geq1$.

\begin{definition}$\mathrm{(cf.\,[HPS18,\,Definition\,19.1]})$
    A \textit{singular Hermitian metric} on $\mathscr{F}$ is a singular Hermitian metric $h$ on the holomorphic vector bundle $F$. 
    We say that a metric is Griffiths semi-positive if $h$ has Griffiths semi-positive on $X(\mathscr{F})$.
\end{definition}

Using the natural inclusion $j:X(\mathscr{F})=X\setminus Z_\mathscr{F}\hookrightarrow X$, we define a natural extension of the $L^2$-subsheaf $\mathscr{E}(h)$ as follows.
Here, $j_*\mathcal{O}_{X\setminus Z_\mathscr{F}}\cong\mathcal{O}_X$ is already known.

\begin{definition}\label{def of ext L2 subsheaf}
    Let $h$ be a singular Hermitian metric on $\mathscr{F}$ which is a singular Hermitian metric on $F$ over $X(\mathscr{F})$. 
    We define the extended natural $L^2$-subsheaf $\mathscr{E}_X(h)$ with respect to $h$ over $X$ by $\mathscr{E}_X(h):=j_*\mathscr{E}(h)$.
\end{definition}


We introduce the definition of the minimal extension property and define Nakano (semi)-positivity with respect to singular Hermitian metrics on torsion-free coherent sheaves.

\begin{definition}$\mathrm{(cf.\,[HPS18,\,Definition\,20.1]})$
    We say that a singular Hermitian metric $h$ on $\mathscr{F}$ has the \textit{minimal extension property} 
    if there exists a nowhere dense closed analytic subset $Z\subseteq X$ with the following two properties:
    \begin{itemize}
        \item [(1)] $\mathscr{F}$ is locally free on $X\setminus Z$, or equivalently, $X\setminus Z\subseteq X(\mathscr{F})$,
        \item [(2)] For every embedding $\iota:B\hookrightarrow X$ with $x=\iota(0)\in X\setminus Z$, and every $v\in F_x$ with $|v|_{h}(x)=1$, 
                    there is a holomorphic section $s\in H^0(B,\iota^*\mathscr{F})$ such that 
        \begin{align*}
            s(0)=v \quad \mathrm{and} \quad \frac{1}{\mathrm{Vol}(B)}\int_B|s|^2_hdV_{B}\leq1,
        \end{align*}
        where $(F,h)$ denotes the restriction to the open subset $X(\mathscr{F})$.
    \end{itemize} 
\end{definition}

\begin{definition}\label{Def loc L2 Nak semi posi on coh sheaf}
    We say that a singular Hermitian metric $h$ on $\mathscr{F}$ is \textit{locally} $L^2$-\textit{type Nakano semi}-\textit{positive}  
    if there exists a nowhere dense closed analytic subset $Z\subseteq X$ with the following two properties:
    \begin{itemize}
        \item [(1)] $\mathscr{F}$ is locally free on $X\setminus Z$, or equivalently, $X\setminus Z\subseteq X(\mathscr{F})$,
        \item [(2)] For any $t\in X\setminus Z$, there exists a open neighborhood $U\subset X\setminus Z$ of $t$ such that a singular Hermitian metric $h$ on $E$ has the optimal $L^2$-estimate on $U$.
        Here $(F,h)$ denotes the restriction to the open subset $X(\mathscr{F})$.
    \end{itemize} 

    In particular, if we can take $Z=X\setminus X(\mathscr{F})$ then we say that $h$ is \textit{full locally} $L^2$-\textit{type Nakano semi}-\textit{positive} on $X(\mathscr{F})$.
\end{definition}

\begin{definition}
    We say that a singular Hermitian metric $h$ on $\mathscr{F}$ is \textit{locally} $L^2$-\textit{type Nakano positive}  
    if there exists a nowhere dense closed analytic subset $Z\subseteq X$ with the following two properties:
    \begin{itemize}
        \item [(1)] $\mathscr{F}$ is locally free on $X\setminus Z$, or equivalently, $X\setminus Z\subseteq X(\mathscr{F})$,
        \item [(2)] For any $t\in X\setminus Z$, there exist a open neighborhood $U\subset X\setminus Z$ of $t$ and a smooth strictly plurisubharmonic function $\psi$ on $U$ 
        such that a singular Hermitian metric $he^{\psi}$ on $E$ has the optimal $L^2$-estimate on $U$.
        Here $(F,h)$ denotes the restriction to the open subset $X(\mathscr{F})$.
    \end{itemize} 

    In particular, if we can take $Z=X\setminus X(\mathscr{F})$ then we say that $h$ is \textit{full locally} $L^2$-\textit{type Nakano positive} on $X(\mathscr{F})$.
\end{definition}

\begin{definition}
    We say that a singular Hermitian metric $h$ on $\mathscr{F}$ is (\textit{globally}) \textit{Nakano semi}-\textit{positive}  
    if there exists a nowhere dense closed analytic subset $Z\subseteq X$ with the following two properties:
    \begin{itemize}
        \item [(1)] $\mathscr{F}$ is locally free on $X\setminus Z$, or equivalently, $X\setminus Z\subseteq X(\mathscr{F})$,
        \item [(2)] $h$ is (globally) Nakano semi-positive on $X\setminus Z$.
    \end{itemize} 
\end{definition}

\section{Nakano positivity of canonical singular Hermitian metric}

\subsection{Canonical singular Hermitian metric on direct image sheaves}

Let $f:X\to Y$ is a projective and surjective holomorphic mapping between two connected complex manifolds, with $\mathrm{dim}\,X=n+m$ and $\mathrm{dim}\,Y=m$, but there may be singular fiber.
Let 
$L\to X$ be a holomorphic line bundle equipped with a pseudo-effective singular Hermitian metric $h$.
In this subsection, we define the canonical singular Hermitian metric on the direct image sheaf $\mathcal{E}:=f_*(\omega_{X/Y}\otimes L\otimes\mathscr{I}(h))$
in the same way as in \cite{HPS18}.

Construct a Hermitian metric of $\mathcal{E}$ over a Zariski-open subset $Y\setminus Z$ where everything is nice, and then to extend it over the bad locus $Z$.
First, we choose a nowhere dense closed analytic subset $Z\subseteq Y$ with the following three properties:
\begin{itemize}
    \item [(1)] The morphism $f$ is submersion over $Y\setminus Z$,
    \item [(2)] Both $\mathcal{E}$ and the quotient sheaf $f_*(\omega_{X/Y}\otimes L)/\mathcal{E}$ are locally free on $Y\setminus Z$,
    \item [(3)] On $Y\setminus Z$, the locally free sheaf $f_*(\omega_{X/Y}\otimes L)$ has the base change property.
\end{itemize}

By the base change theorem, the third condition will hold as long as the coherent sheaves $R^if_*(\omega_{X/Y}\otimes L)$ are locally free on $Y\setminus Z$. 
The restriction of $\mathcal{E}$ to the open subset $Y\setminus Z$ is a holomorphic vector bundle $E$ of some rank $r\geq1$. 
The second and third condition together guarantee that 
\begin{align*}
    E_t:=\mathcal{E}|_t\subseteq f_*(\omega_{X/Y}\otimes L)|_t=H^0(X_t,\omega_{X_t}\otimes L|_{X_t})
\end{align*}
whenever $t\in Y\setminus Z$.

\begin{lemma}$\mathrm{(cf.\,[HPS18,\,Lemma\,22.1]})$
    For any $t\in Y\setminus Z$, we have inclusions 
    \begin{align*}
        H^0(X_t,\omega_{X_t}\otimes L|_{X_t}\otimes\mathscr{I}(h|_{X_t}))\subseteq E_t\subseteq H^0(X_t,\omega_{X_t}\otimes L|_{X_t}).
    \end{align*}
\end{lemma}

Here, we can immediately see that the two subspaces 
\begin{align*}
    H^0(X_t,\omega_{X_t}\otimes L|_{X_t}\otimes\mathscr{I}(h|_{X_t}))\subseteq E_t
\end{align*}
are equal for almost everywhere $t\in Y\setminus Z$. 
But unless $\mathcal{E}=0$, the two subspaces are different for example at points where $h|_{X_t}$ is identically equal to $+\infty$.

On each $E_t$ with $t\in Y\setminus Z$, we can define a singular Hermitian metric $H$ as follows.
For any element $\alpha\in E_t\subseteq H^0(X_t,\omega_{X_t}\otimes L|_{X_t})$,
we can integrate over the compact complex manifold $X_t$ and define the inner product of $\alpha$ with respect to $H$ by
\begin{align*}
    |\alpha|^2_H(t):=\int_{X_t}|\alpha|^2_h\in[0,+\infty].
\end{align*}
Clearly $|\alpha|_H(t)<+\infty$ if and only if $\alpha\in H^0(X_t,\omega_{X_t}\otimes L|_{X_t}\otimes\mathscr{I}(h|_{X_t}))$.
By Ehresmann's fibration theorem and Fubini's theorem, the function $t\mapsto |s|_H(t)$ is measurable for any local holomorphic section $s$ of $E$.

From the discussion in \cite{HPS18}, the singular Hermitian metric $H$ over $Y\setminus Z$ is well-defined on the entire open set $Y(\mathcal{E})$.
Then we say that this extended metric $H$ on $E$ over $Y(\mathcal{E})$ is a \textit{canonical singular Hermitian metric} of $\mathcal{E}$.
Finally, we define the following. 

\begin{definition}
    We define the set $\Sigma_H$ on $Y$ related to the unbounded-ness of $H$ by
    \begin{align*}
        \Sigma_H:=\{t\in Y\mid \mathcal{E}_t\subsetneq H^0(X_t,K_{X_t}\otimes L|_{X_t})\}.
    \end{align*}
\end{definition}

Here, for any $t\in Y\setminus Z$ if $\mathscr{I}(h|_{X_t})=\mathcal{O}_{X_t}$ then $t\notin\Sigma_H$ and $H(t)$ is bounded by $\int_{X_t}e^{-\varphi}<+\infty$, where $h=e^{-\varphi}$ on local.
And for almost everywhere $t\notin\Sigma_H$, we get $H^0(X_t,K_{X_t}\otimes L|_{X_t}\otimes\mathscr{I}(h|_{X_t}))=\mathcal{E}_t=H^0(X_t,K_{X_t}\otimes L|_{X_t})$.
Let $\Sigma_h:=\{t\in Y\mid\int_{X_t}e^{-\varphi}=+\infty\}$ be a set related to the unbounded-ness of $h$, then we have that $\Sigma_H\setminus Z\subseteq\Sigma_h\setminus Z$.

\subsection{An approach of globally Nakano semi-positivity of $H$}

We consider the case where $X$ is projective.
By projectivity of $X$, there exist two hypersurfaces $Z_1$ and $Z_2$ such that $X\setminus Z_1$ is Stein and that $S:=X\setminus (Z_1\cup Z_2)$ is also Stein and $L|_S$ is trivial.
Let $\varphi:=-\log h|_S$ then $\varphi$ is plurisubharmonic function on $S$ and $h=e^{-\varphi}$ on $S$.
By [FN80,\,Theorem\,5.5], there exists a sequence of smooth plurisubharmonic functions $(\varphi_\nu)_{\nu\in\mathbb{N}}$ on $S$ decreasing to $\varphi$ a.e. pointwise.
Here, there is a smooth exhaustive strictly plurisubharmonic function $\psi$ on $S$ such that $\sup_S\psi=+\infty$. 
Let $S_\nu:=\{z\in S\mid\psi(z)<1/\nu\}$ be a Stein sublevel set.

Let $S_t^\nu:=X_t\cap S_\nu$ be Stein subsets. We define the Hermitian metric $H_\nu$ on $E$ over $Y\setminus Z$ by for any elements $u,v\in E_t$, 
\begin{align*}
    (u,v)_{H_\nu}(t):=\int_{\overline{S_t^\nu}}u\wedge\overline{v}e^{-\varphi_\nu}=\int_{S_t^\nu}u\wedge\overline{v}e^{-\varphi_\nu}.
\end{align*}
Then $H_\nu$ is smooth by closed-ness of $\overline{S_t^\nu}$, Ehresmann's theorem and Fubini's theorem.

\begin{question}
    Is this smooth Hermitian metric $H_\nu$ 
is Nakano semi-positive?
\end{question}

\begin{remark}
    If the Question is correct, then $H$ has (globally) Nakano semi-positivity.
    In fact, $(H_\nu)_{\nu\in\mathbb{N}}$ is a sequence of smooth Nakano semi-positive Hermitian metrics increasing to $H$ pointwise a.e., and we can use the following proposition.
\end{remark}

\begin{proposition}$\mathrm{(cf.\,[Ina22,\,Proposition\,6.1]})$
    Let $h$ be a singular Hermitian metric on a holomorphic vector bundle. 
    If there exists a sequence of smooth Nakano semi-positive metrics $(h_\nu)_{\nu\in\mathbb{N}}$ increasing to $h$ pointwise a.e.,
    then $h$ is (globally) Nakano semi-positive.
\end{proposition}

In the same way as above, extend the smooth Hermitian metric $H_\nu$ on $E$ to a smooth Hermitian metric $\widetilde{H}_\nu$ on $F=f_*(\omega_{X/Y}\otimes L)$ over $Y$.
Fixed $\nu$, let 
\begin{align*}
    G_t^\nu:=\{f\in H^0(S_t^\nu,K_{S_t^\nu})\mid \int_{S_t^\nu}|f|^2e^{-\varphi_\nu}<+\infty\}
\end{align*}
be fibers where $K_{S_t^\nu}=K_{X_t}|_{S_\nu}$ and $L|_{S_\nu}$ is trivial. 
Then there is a natural inclusions $F_t=H^0(X_t,K_{X_t}\otimes L)\hookrightarrow G_t^\nu$. 
We define a infinite vector bundle $G^\nu:=\bigcup_{t\in Y}G_t^\nu\to Y$ and a Hermitian metric $H_{G^\nu}$ by for any $f,g\in G^\nu_t$,
\begin{align*}
    (f,g)_{H_{G^\nu}}(t):=\int_{S_t^\nu}f\wedge\overline{g}e^{-\varphi_\nu}.
\end{align*} 
Here $F$ is a natural subbundle of $G^\nu$ and $f|_{S_\nu}:S_\nu\to Y$ is Stein fibration.

\begin{remark}
    From Berndtsson-P\u{a}un's work \cite{BP08}, for any subset $U\subset Y\setminus Z$ such that $K_U$ is trivial we have that the relative Bergman kernel of $G^\nu|_U$ to Stein fibrations is plurisubharmonic. Hence, $H_{G^\nu}$ is Griffiths semi-positive.
\end{remark}

\begin{question}
    Does this Hermitian metric $H_{G^\nu}$ on $G^\nu$ have Nakano semi-positivity?
    And if this is true, does $H_{G^\nu}$ induce Nakano semi-positivity of smooth Hermitian metrics $H_\nu$ and $\widetilde{H}_\nu$ on $E=\mathcal{E}|_{Y(\mathcal{E})}$ and $F=f_*(\omega_{X/Y}\otimes L)$ respectively?
\end{question}

\subsection{Locally $L^2$-type Nakano (semi)-positivity of $H$}

Let $f:X\to Y$ be a projective surjective morphism between two connected complex manifolds and $L$ be a holomorphic vector bundle on $X$ equipped with a pseudo-effective Hermitian metric $h$.
For the canonical singular Hermitian metric $H$ of the direct image sheaf $\mathcal{E}=f_*(\omega_{X/Y}\otimes L\otimes\mathscr{I}(h))$, the following theorem is known with respect to the positivity property.

\begin{theorem}\label{HPS18, Theorem21.1}$\mathrm{(cf.\,[HPS18,\,Theorem\,21.1]})$
    The direct image sheaf $\mathcal{E}=f_*(\omega_{X/Y}\otimes L\otimes\mathscr{I}(h))$ has a canonical singular Hermitian metric $H$.
    This metric is Griffiths semi-positive and satisfies the minimal extension property.
\end{theorem}

In this subsection, we show that this metric $H$ on $\mathcal{E}$ has locally Nakano (semi)-positivity.
This proof is inspired by the proof of the smooth case using $L^2$-estimates in [DNWZ22,\,Theorem\,1.6]. 

\begin{theorem}\label{H of full loc L2 Nak semi-posi in subsection}
    Let $H$ be a canonical singular Hermitian metric on $\mathcal{E}=f_*(\omega_{X/Y}\otimes L\otimes\mathscr{I}(h))$ which induced by $h$.
    If $X$ is projective and there exists an analytic set $A$ such that $\Sigma_H\subseteq A$ then $H$ is full locally $L^2$-type Nakano semi-positive on $Y(\mathcal{E})$.
\end{theorem}

\begin{proof}
    First, we prove that $H$ is locally $L^2$-type Nakano semi-positive, i.e. for any $t\in Y\setminus Z$, there exists a open neighborhood $U\subset Y\setminus Z$ of $t$ such that $H$ has the optimal $L^2$-estimate on $U$.
    We can take $U$ to Stein. Let $X_U=\pi^{-1}(U)$ and $X_U^\nu:=X_U\cap S_\nu$ then $X_U^\nu$ is also Stein by $f$ is holomorphic.
    Let $\overline{\partial}$-closed $g\in\mathscr{D}^{m,1}(U,E)$ and $\psi$ be any smooth strictly plurisubharmonic function on $U$.
    We can write $g(t)=\sum^m_{j=1} g_j(t)d\overline{t}_j\wedge dt$, with $g_j(t)\in E_t\subseteq H^0(X_t,\omega_{X_t}\otimes L|_{X_t})$.
    We can identify $g$ as a smooth compact supported $(n+m,1)$-form $\widetilde{g}(t,z):=\sum^m_{j=1}g_j(t,z)d\overline{t}_j\wedge dt$ on $X$, with $g_j(t,z)$ begin holomorphic section $\omega_{X_t}\otimes L|_{X_t}$.
    We have the following observations:
    \begin{itemize}
        \item $\overline{\partial}_zg_j(t,z)=0$ for any fiexd $t\in U$, since $g_j(t,z)$ are holomorphic sections of $\omega_{X_t}\otimes L|_{X_t}$,
        \item $\overline{\partial}_tg_j=0$, since $g$ is a $\overline{\partial}$-closed form on $U$.
    \end{itemize}

    We consider the integration 
    \begin{align*}
        &\int_{X_U^\nu}\langle[i\Theta_{L,h_\nu}+\idd f^*\psi\otimes\mathrm{id}_L,\Lambda_\omega]^{-1}\widetilde{g},\widetilde{g}\rangle_{h_\nu,\omega}e^{-f^*\psi}dV_\omega\\
        &=\int_{X_U^\nu}\langle[\idd\varphi_\nu+\idd f^*\psi,\Lambda_\omega]^{-1}\widetilde{g},\widetilde{g}\rangle_\omega e^{-\varphi_\nu-f^*\psi}dV_\omega.
    \end{align*}
    Note that, acting on $\Lambda^{n+m,1}T^*_X\otimes L$, we have 
    \begin{align*}
        [i\Theta_{L,h_\nu}+\idd f^*\psi,\Lambda_\omega]\geq[\idd f^*\psi,\Lambda_\omega]\geq0.  \tag*{($\ast$)}
    \end{align*}

    We take a local coordinate $(t_1,\ldots,t_m,z_1,\ldots,z_n)$ on $X$ near $t$, where $t_1,\ldots,t_m$ is the standard coordinate on $U\subset\mathbb{C}^m$. 
    Let $\omega'=i\sum^m_{j=1}dt_j\wedge d\overline{t}_j+i\sum^n_{j=1}dz_j\wedge d\overline{z}_j$ and $\omega_0=i\sum^m_{j=1}dt_j\wedge d\overline{t}_j$.
    Note that 
    \begin{align*}
        \idd f^*\psi=\sum^m_{j,k=1}\frac{\partial^2\psi}{\partial t_j\partial\overline{t}_k}dt_j\wedge d\overline{t}_k, 
    \end{align*}
    we have that 
    \begin{align*}
        [\idd f^*\psi,\Lambda_{\omega'}]\widetilde{g}&=\sum_{j,k}\frac{\partial^2\psi}{\partial t_j\partial\overline{t}_k}g_j(t,z)dt \wedge d\overline{t}_k,\\
        [\idd f^*\psi,\Lambda_{\omega'}]^{-1}\widetilde{g}&=\sum_{j,k}\psi^{jk}g_j(t,z)dt \wedge d\overline{t}_k,
    \end{align*}
    at $t$, where $(\psi^{jk})=(\frac{\partial^2\psi}{\partial t_j\partial\overline{t}_k})^{-1}$.
    By Lemma \ref{DNWZ22, Lemma4.7}, we have 
    \begin{align*}
        \langle[\idd f^*\psi,\Lambda_\omega]^{-1}\widetilde{g},\widetilde{g}\rangle_\omega dV_\omega
        &=\langle[\idd f^*\psi,\Lambda_{\omega'}]^{-1}\widetilde{g},\widetilde{g}\rangle_{\omega'} dV_{\omega'}\\
        &=\sum_{j,k}\psi^{jk}c_ng_j\wedge\overline{g}_kc_mdt\wedge d\overline{t}.
    \end{align*}

    By Fubini's theorem, we get that 
    \begin{align*}
        \int_{X_U^\nu}\langle[\idd f^*\psi,\Lambda_\omega]^{-1}\widetilde{g},\widetilde{g}\rangle_\omega e^{-\varphi_\nu-\pi^*\psi}dV_\omega
        &=\int_{X_U^\nu}\sum_{j,k}\psi^{jk}c_ng_j\wedge\overline{g}_ke^{-\varphi_\nu-f^*\psi}c_mdt\wedge d\overline{t}\\
        &=\int_U(g_j,g_k)_{H_\nu}(t)\psi^{jk}e^{-\psi}c_mdt\wedge d\overline{t}\\
        &=\int_U\langle[\idd\psi,\Lambda_{\omega_0}]^{-1}g,g\rangle_{H_\nu,\omega_0}e^{-\psi}dV_{\omega_0}\\
        &\leq\int_U\langle[\idd\psi,\Lambda_{\omega_0}]^{-1}g,g\rangle_{H,\omega_0}e^{-\psi}dV_{\omega_0}\\
        &<+\infty.
    \end{align*}

    From H\"ormander's $L^2$-estimate, i.e. Theorem \ref{Hormander L2-estimate}, there is solution $\widetilde{v}_\nu\in L^2_{n+m,0}(X_U^\nu,L,h_\nu,\omega)$ such that $\overline{\partial}\widetilde{v}_\nu=\widetilde{g}$ on $X_U^\nu$ and satisfies the following estimate 
    \begin{align*}
        \int_{X_U^\nu}|\widetilde{v}_\nu|^2_{h_\nu}e^{-f^*\psi}dV_\omega&=\int_{X_U^\nu}c_{n+m}\widetilde{v}_\nu\wedge\overline{\widetilde{v}}_\nu e^{-\varphi_\nu-f^*\psi}\\
        &\leq\int_{X_U^\nu}\langle[i\Theta_{L,h_\nu}+\idd f^*\psi\otimes\mathrm{id}_L,\Lambda_\omega]^{-1}\widetilde{g},\widetilde{g}\rangle_{h_\nu,\omega}e^{-f^*\psi}dV_\omega\\
        &\leq\int_U\langle[\idd\psi,\Lambda_{\omega_0}]^{-1}g,g\rangle_{H,\omega_0}e^{-\psi}dV_{\omega_0}<+\infty.
    \end{align*}

    We observe that $\overline{\partial}\widetilde{v}|_{S_t^\nu}=0$ for any fiexd $t\in U$, since $\overline{\partial}\widetilde{v}=\widetilde{g}$ on $X_U^\nu$ where $S_t^\nu:=X_t\cap S_\nu$.
    From the monotonicity to $\nu$ of $|\bullet|^2_{h_\nu}$ by increasing $(h_\nu)_{\nu\in\mathbb{N}}$, the family $(\widetilde{v}_\nu)_{\nu_1\leq\nu\in\mathbb{N}}$ forms a bounded sequence in $L^2_{n+m,0}(X_U^{\nu_1},L,h_{\nu_1},\omega)$.
    Therefore, we can obtain a weakly convergence subsequence in $L^2_{n+m,0}(X_U^{\nu_1},L,h_{\nu_1},\omega)$. 
    By using a diagonal argument, we get a subsequence $(\widetilde{v}_{\nu_k})_{k\in\mathbb{N}}$ of $(\widetilde{v}_\nu)_{\nu_1\leq\nu\in\mathbb{N}}$ converging weakly in $L^2_{n+m,0}(X_U^{\nu_1},L,h_{\nu_1},\omega)$ for any $\nu_1$,
    where $\widetilde{v}_{\nu_k}\in L^2_{n+m,0}(X_U^{\nu_k},L,h_{\nu_k},\omega)\subset L^2_{n+m,0}(X_U^{\nu_1},L,h_{\nu_1},\omega)$. 

    We denote by $\widetilde{v}$ the weakly limit of $(\widetilde{v}_{\nu_k})_{k\in\mathbb{N}}$. Then $\widetilde{v}$ satisfies $\overline{\partial}\widetilde{v}=\widetilde{g}$ on $X_U$ and 
    \begin{align*}
        \int_{X_U^{\nu_k}}|\widetilde{v}|^2_{h_{\nu_k}}e^{-f^*\psi}dV_\omega\leq\int_U\langle[\idd\psi,\Lambda_{\omega_0}]^{-1}g,g\rangle_{H,\omega_0}e^{-\varphi}dV_{\omega_0}<+\infty
    \end{align*}
    for any $k\in\mathbb{N}$. Taking weakly limit $k\to+\infty$ and using the monotone convergence theorem, we have the following estimate 
    \begin{align*}
        \int_{X_U}|\widetilde{v}|^2_he^{-f^*\psi}dV_\omega&=\int_{X_U\setminus (Z_1\cup Z_2)}|\widetilde{v}|^2_he^{-f^*\psi}dV_\omega\\
        &\leq\int_U\langle[\idd\psi,\Lambda_{\omega_0}]^{-1}g,g\rangle_{H,\omega_0}e^{-\psi}dV_{\omega_0}<+\infty,
    \end{align*}
    i.e. $\widetilde{v}\in L^2_{n+m,0}(X_U,L,h,\omega)$. 

    Here we write $\widetilde{v}(t,z)=\widetilde{V}(t,z)dz\wedge dt$, then 
    $\frac{\partial\widetilde{V}}{\partial\overline{z}_j}=0$, i.e. $\overline{\partial}\widetilde{v}|_{X_t}=0$ for any fiexd $t\in U$, since $\overline{\partial}\widetilde{v}=\widetilde{g}$ on $X_U$. 
    This means that $\widetilde{V}(t,\cdot)dz\in H^0(X_t,\omega_{X_t}\otimes L|_{X_t})$. 
    We can identify $\widetilde{v}$ as a $(m,0)$-form $v(t):=V(t)dt$ on $U$, with $V(t)=\widetilde{V}(t,\cdot)dz\in H^0(X_t,\omega_{X_t}\otimes L|_{X_t})$.

    From Fubini's theorem, we have that 
    \begin{align*}
        \int_{X_U}|\widetilde{v}|^2_he^{-f^*\psi}dV_\omega=\int_{X_U}c_{n+m}\widetilde{v}\wedge\overline{\widetilde{v}}he^{-f^*\psi}=\int_U||v||^2_He^{-\psi}dV_{\omega_0}.
    \end{align*}
    Therefore, we get 
    \begin{align*}
        \int_U||v||^2_{H,\omega_0}e^{-\psi}dV_{\omega_0}\leq\int_U\langle[\idd\psi,\Lambda_{\omega_0}]^{-1}g,g\rangle_{H,\omega_0}e^{-\psi}dV_{\omega_0}<+\infty.
    \end{align*}
    Here, by boundedness of the integral of $||v||^2_H$, for any almost everywhere $t\in U$ we have that $||v||^2_H(t)<+\infty$, i.e. $V(t)\in H^0(X_t,K_{X_t}\otimes L|_{X_t}\otimes\mathscr{I}(h|_{X_t}))\subseteq\mathcal{E}_t$.

    Form the assumption $\Sigma_H\subseteq A$, replacing $v=0$, i.e. $V=0$, on $A$ then for any $t\in U$ we get $V(t)\in H^0(X_t,K_{X_t}\otimes L|_{X_t})=\mathcal{E}_t$. 
    By the Lebesgue measure of $A$ is zero, this means that $v\in L^2_{m,0}(U,E,H,\omega_0)$ and $\overline{\partial}v=g$ on $U\setminus A$.
    From Lemma \ref{Ext d-equation for hypersurface}, we get $\overline{\partial}v=g$ on $U$.
    Hence, we showed that $H$ satisfies the optimal $L^2$-estimate on $U$.

    Finally, we prove that then $H$ is full locally $L^2$-type Nakano semi-positive on $Y(\mathcal{E})$.
    Put $Z_\mathcal{E}:=Y\setminus Y(\mathcal{E})$ then $Z_\mathcal{E}\subseteq Z$ and there is a analytic set $B$ such that $Z=Z_\mathcal{E}\cup B$.
    Therefore, it is sufficient to show that for any $t\in B\setminus Z_\mathcal{E}$, there exists a open neighborhood $U\subset Y(\mathcal{E})$ of $t$ such that $H$ has the optimal $L^2$-estimate on $U$.
    This can be shown in the same way as above by using Lemma \ref{Ext d-equation for hypersurface}.
\end{proof}

\begin{lemma}\label{Ext d-equation for hypersurface}$\mathrm{(cf.\,[Dem82,\,Lemma\,6.9],\,[Ber10,\,Lemma\,5.1.3])}$ 
    Let $\Omega$ be an open subset of $\mathbb{C}^n$ and $Z$ be a complex analytic subset of $\Omega$. Assume that $u$ is a $(p,q-1)$-form with $L^2_{loc}$ coefficients and $g$ is a $(p,q)$-form with $L^1_{loc}$ coefficients such that $\overline{\partial}u=g$ on $\Omega\setminus Z$ (in the sense of currents).
    Then $\overline{\partial}u=g$ on $\Omega$.
\end{lemma}

\begin{lemma}\label{complete kah on X_j-Z}$\mathrm{(cf.\,[Dem82,\,Theorem\,1.5]})$ 
    Let $X$ be a \kah manifold and $Z$ be an analytic subset of $X$. Assume that $\Omega$ is a relatively open subset of $X$ possessing a complete \kah metric. Then $\Omega\setminus Z$ carries a complete \kah metric.
\end{lemma}

By using Lemma \ref{complete kah on X_j-Z} and Demailly's approximation theorem (see \cite{Dem94}), the following can be shown similarly as above.
Here, we do not use Demailly's approximation theorem in the proof of Theorem \ref{H of full loc L2 Nak semi-posi in subsection} because the left term of $(\ast)$ is not necessarily semi-positive and H\"ormander's $L^2$-estimate cannot be used.

\begin{theorem}\label{H of full loc L2 Nak posi in subsection}
    Let $H$ be a canonical singular Hermitian metric on $\mathcal{E}=f_*(\omega_{X/Y}\otimes L\otimes\mathscr{I}(h))$ which induced by $h$.
    We assume that $X$ is compact \kah and $h$ is big.
    If there exists an analytic set $A$ such that $\Sigma_H\subseteq A$ then the $H$ is full locally $L^2$-type Nakano positive on $Y(\mathcal{E})$.
\end{theorem}

Here, the $L^2$-subsheaf $\mathscr{E}(H)$ of $H$ is a subsheaf of $E=\mathcal{E}|_{Y(\mathcal{E})}$ over $Y(\mathcal{E})$.
For a natural inclusion $j:Y\setminus Z_\mathcal{E}=Y(\mathcal{E})\hookrightarrow Y$, the natural extended $L^2$-subsheaf with respect to $H$ over $Y$ is defined by $\mathscr{E}_Y(H):=j_*\mathscr{E}(H)$ as in Definition \ref{def of ext L2 subsheaf}.

\begin{theorem}
    Let $f:X\to Y$ be a projective and surjective holomorphic mapping between two connected complex manifolds and $L$ be a holomorphic line bundle on $X$ equipped with a pseudo-effective metric $h$.
    Let $H$ be a canonical singular Hermitian metric on $f_*(\omega_{X/Y}\otimes L\otimes\mathscr{I}(h))$.
    If $X$ is projective and there exists an analytic set $A$ such that $\Sigma_H\subseteq A$ then the natural extended $L^2$-subsheaf $\mathscr{E}_Y(H)$ over $Y$ is coherent.
\end{theorem}

\begin{proof}
    From Theorem \ref{H of full loc L2 Nak semi-posi in subsection} and [Ina22,\,Proposition\,4.4], we have that the $L^2$-subsheaf $\mathscr{E}(H)$ over $Y(\mathcal{E})$ is coherent.
    For the natural inclusion $j:Y\setminus Z_\mathcal{E}=Y(\mathcal{E})\hookrightarrow Y$, we are already known $j_*\mathcal{O}_{Y\setminus Z_\mathcal{E}}\cong\mathcal{O}_Y$ since the analytic set $Z_\mathcal{E}:=Y\setminus Y(\mathcal{E})$ is codimension $\geq2$.
    By Riemann's extension theorem, the sheaf $j_*\mathscr{E}(H)=\mathscr{E}_Y(H)$ is also coherent.
\end{proof}

\begin{corollary}
    Let $H$ be a canonical singular Hermitian metric on $\mathcal{E}=f_*(\omega_{X/Y}\otimes L\otimes\mathscr{I}(h))$ which induced by a pseudo-effective metric $h$ on $L$.
    Let $B_H\subseteq Y(\mathcal{E})\setminus \Sigma_H$ be a open subset. Here, $\mathcal{E}|_{Y(\mathcal{E})}=E$ is holomorphic vector bundle.
    If $X$ is projective then for any local open subset $U\subset B_H$, $(E,H)$ satisfies the optimal $L^2$-estimate on $U$. And the $L^2$-subsheaf $\mathscr{E}_Y(H)$ is coherent on $B_H$.
\end{corollary}

\begin{remark}
    This theorem and corollary hold even if the situation is that $X$ is compact \kah and $h$ is big by Theorem \ref{H of full loc L2 Nak posi in subsection}.
\end{remark}

\begin{corollary}
    Let $\mathscr{F}$ be a torsion-free coherent sheaf on complex manifold $X$ equipped with a singular Hermitian metric $h$. 
    If $h$ is full locally $L^2$-type Nakano semi-positive on $X(\mathscr{F})$ then the natural extended $L^2$-subsheaf $\mathscr{E}_X(h)$ is coherent.
\end{corollary}

\section{The minimal extension property and Nakano semi-positivity}

In this section, we study the relation between the minimal extension property and Nakano semi-positivity, and prove the following theorem.
For holomorphic line bundles, the two properties are equivalent from the optimal Ohsawa-Takegoshi $L^2$-extension theorem (see \cite{Blo13}, \cite{GZ12}) and the proof of [HPS18,\,Theorem\,21.1].
In the case of holomorphic vector bundles, the Ohsawa-Takegoshi $L^2$-extension theorem follows from Nakano semi-positivity, so it is likely to have the minimal extension property if it is Nakano semi-positive.
However, it turns out that in general the converse does not hold true. 
This phenomenon is first mentioned in \cite{HI20} for the positivity called \textit{weak Ohsawa}-\textit{Takegoshi} in a close concept instead of the minimal extension property.

\begin{theorem}\label{min ext prop and not Nak semi-posi}
    Let $\mathscr{F}$ be a torsion-free coherent sheaf on a complex manifold $X$. 
    Even if $\mathscr{F}$ has a singular Hermitian metric satisfying the minimal extension property, it does not necessarily have a singular Hermitian metric $h$ which is (globally) Nakano semi-positive and satisfying $\nu(-\log\mathrm{det}\,h,x)<2$ for any point $x\in X(\mathscr{F})$.
\end{theorem}

Here, this symbol $\nu$ denotes the Lelong number and is defined by 
\begin{align*}
    \nu(\varphi,x):=\liminf_{z\to x}\frac{\varphi(z)}{\log|z-x|}
\end{align*}
for a plurisubharmonic function $\varphi$ and some coordinate $(z_1,\ldots,z_n)$ around $x$.
And it is already known that if $\nu(-\log\mathrm{det}\,h,x)<2$ then $\mathscr{E}(h)_x=\mathcal{O}(E)_x$.

\subsection{Exact sequences of torsion-free coherent sheaves and Positivity}

Consider the inheritance of positivity in exact sequences.
The following is already known for the minimal extension property.

\begin{proposition}\label{LS22 Prop 6 and 7}$\mathrm{(cf.\,[LS22,\,Proposition\,6\,and\,7]})$
    Let 
    \begin{align*}
        0\longrightarrow \mathscr{S} \stackrel{j}{\hookrightarrow} \mathscr{F} \stackrel{g}{\twoheadrightarrow} \mathscr{Q} \longrightarrow 0
    \end{align*}
    be an exact sequence of torsion-free coherent sheaves and $h$ be a singular Hermitian metric on $\mathscr{F}$ which has the minimal extension property. Then we have the following
    \begin{itemize}
        \item [$(a)$] If $j$ is generically an isomorphism, then $h$ extends to a singular Hermitian metric $h_{\mathscr{G}}$ on $\mathscr{G}$ satisfying the minimal extension property,
        \item [$(b)$] The induced metric $h_{\mathscr{Q}}$ has also the minimal extension property.
    \end{itemize}
\end{proposition}

For Griffiths and Nakano positivity of smooth metrics, the following is known.

\begin{proposition}\label{Dem-book Prop6.10}$\mathrm{(cf.\,[Dem}$-$\mathrm{book,\,ChapterVII,\,Proposition\,6.10]})$
    Let $0\to S\to E\to Q\to 0$ be an exact sequence of hermitian vector bundles. Then we have the following 
    
    \!\!\!\!\!\!\!$(a) \,\,E\geq_{Grif}0 \Longrightarrow Q\geq_{Grif}0$, $(b) \,\,E\leq_{Grif}0 \Longrightarrow S\leq_{Grif}0$, $(c) \,\,E\leq_{Nak}0 \Longrightarrow S\leq_{Nak}0$,
    and analogous implications hold true for strictly positivity. 

    In particular, a Nakano semi-positive metric of $E$ does not necessarily induce a Nakano semi-positive metric of $Q$.
\end{proposition}

Here, for the inheritance of semi-positivity from $E$ to $Q$, Nakano semi-positivity has a counterexample (see Proposition \ref{Q has no positivity}), but by rephrasing condition $(c)$, we find the following with respect to dual Nakano positivity.

\begin{corollary}\label{exact sequence and smooth dual Nakano semi posi}
    Let $g:E\twoheadrightarrow Q$ be a quotient onto a holomorphic vector bundle. Then if $E$ is dual Nakano (semi)-positive then $Q$ is also dual Nakano (semi)-positive.
\end{corollary}

\begin{proof}
    There exists a holomorphic vector bundle $S$ such that $0\to S\to E\to Q\to 0$ is an exact sequence of holomorphic vector bundles. Then the sequence $0\to Q^*\to E^*\to S^*\to 0$ is also exact.
    Here, $E^*$ is Nakano (semi)-negative by the assumption. By $(c)$ of Proposition \ref{Dem-book Prop6.10}, $Q^*$ is Nakano (semi)-negative.
\end{proof}

We consider the positivity of singular Hermitian metrics. 
For Griffiths positivity, [HPS18,\,Proposition\,19.3] is already known, and we obtain the following proposition for (dual) Nakano positivity.

\begin{proposition}$\mathrm{(cf.\,[HPS18,\,Proposition\,19.3]})$
    Let $\phi:\mathscr{F}\to\mathscr{G}$ be a morphism between two torsion-free coherent sheaves that is generically an isomorphism. 
    If $\mathscr{F}$ has a singular Griffiths semi-positive Hermitian metric, then so does $\mathscr{G}$.
\end{proposition}

\begin{proposition}
    Let $0\to S\to E\to Q\to 0$ be an exact sequence of holomorphic vector bundles. Let $h$ be a singular Hermitian metric on $E$. Then we have that 
    \begin{itemize}
        \item [$(a)$] If $h$ is Nakano semi-negative then $S$ has a natural induced singular Hermitian metric which is Nakano semi-negative.
        \item [$(b)$] If $h$ is dual Nakano semi-positive then $Q$ has a natural induced singular Hermitian metric which is dual Nakano semi-positive.
    \end{itemize}
\end{proposition}

In particular, Proposition \ref{Dem-book Prop6.10} and Corollary \ref{exact sequence and smooth dual Nakano semi posi} follow from this proposition.

\begin{proof}
    $(a)$ We define the natural singular Hermitian metric $h_S$ of $S$ induced from $h$ by $|u|_{h_S}:=|ju|_h$ for any section $u$ of $S$.
    By the assumption, for any local holomorphic section $s_j\in\mathcal{O}(E)$, the $(n-1,n-1)$-form $T^h_u=\sum(s_j,s_k)_h\widehat{dz_j\wedge d\overline{z}_k}$ is plurisubharmonic, i.e. $\idd T^h_u\geq0$.
    For any local holomorphic section $u_k\in\mathcal{O}(S)$, images $ju_k$ is also local holomorphic section of $E$, i.e. $ju_k\in\mathcal{O}(E)$.
    Then from the equality
    \begin{align*}
        T^{h_S}_u=\sum(u_j,u_k)_{h_S}\widehat{dz_j\wedge d\overline{z}_k}=\sum(ju_j,ju_k)_h\widehat{dz_j\wedge d\overline{z}_k}=T^h_{ju},
    \end{align*}
    we have that $T^{h_S}_u$ is also plurisubharmonic, i.e. $h_S$ is Nakano semi-negative.

    $(b)$ Here, the sequence $0\to Q^*\to E^*\to S^*\to 0$ is also exact. By the assumption and $(a)$, $Q^*$ has a Nakano semi-negative singular Hermitian metric.
\end{proof}

\subsection{A concrete example}

We consider the following exact sequence of holomorphic vector bundles
\begin{align*}
    0\longrightarrow \mathcal{O}_{\mathbb{P}^n}(-1) \stackrel{j}{\hookrightarrow} \underline{V}:=\mathbb{P}^n\times\mathbb{C}^{n+1} \stackrel{g}{\twoheadrightarrow} Q:=\underline{V}/\mathcal{O}_{\mathbb{P}^n}(-1) \longrightarrow 0.
\end{align*}
From this sequence, we get $\mathrm{det}\,\underline{V}=\mathrm{det}\,Q\otimes \mathcal{O}_{\mathbb{P}^n}(-1)$ and get isomorphisms 
\begin{align*}
    \mathrm{det}\,Q\cong \mathcal{O}_{\mathbb{P}^n}(1), \quad T_{\mathbb{P}^n}=Q\otimes \mathcal{O}_{\mathbb{P}^n}(1)\cong Q\otimes \mathrm{det}\,Q,
\end{align*}
where $\mathrm{det}\,\underline{V}$ is also trivial. 
By Griffiths semi-positivity of $\underline{V}$ and Corollary \ref{exact sequence and smooth dual Nakano semi posi}, the bundle $Q$ is dual Nakano semi-positive and then Griffiths semi-positive.
Therefore, $T_{\mathbb{P}^n}$ is Nakano semi-positive from Demailly-Skoda's theorem (see \cite{DS80}), and is Griffiths positive from $Q\geq_{Grif}0$ and $\mathrm{det}\,Q\cong \mathcal{O}_{\mathbb{P}^n}(1)>0$.
But the tangent bundle $T_{\mathbb{P}^n}$ has no smooth Nakano positive metric. In fact, if $T_{\mathbb{P}^n}>_{Nak}0$ then $H^q(\mathbb{P}^n,K_{\mathbb{P}^n}\otimes T_{\mathbb{P}^n})=0$ for any $q\geq1$ by the Nakano vanishing theorem.
However, this contradicts the following
\begin{align*}
    H^{n-1}(\mathbb{P}^n,K_{\mathbb{P}^n}\otimes T_{\mathbb{P}^n})\cong H^1(\mathbb{P}^n,T^*_{\mathbb{P}^n})=H^{1,1}(\mathbb{P}^n,\mathbb{C})=\mathbb{C}.
\end{align*}

\begin{proposition}\label{Q has no positivity}
    We have that $Q$ has no smooth Griffiths positive Hermitian metric and no singular Hermitian metric which is (globally) Nakano semi-positive and satisfying $\nu(-\log\mathrm{det}\,h,x)<2$ for any point $x\in \mathbb{P}^n$.
\end{proposition}

\begin{proof}
    First, if $Q$ has a smooth Griffiths positive Hermitian metric then $T_{\mathbb{P}^n}\cong Q\otimes\mathrm{det}\,Q$ has a smooth Nakano positive Hermitian metric by Demailly-Skoda's theorem.
    Second, if $Q$ has a smooth Nakano semi-positive Hermitian metric then $T_{\mathbb{P}^n}\cong Q\otimes\mathrm{det}\,Q$ has a smooth Nakano positive Hermitian metric by $\mathrm{det}\,Q\cong\mathcal{O}_{\mathbb{P}^n}(1)$ is positive line bundle.
    But these contradict that $T_{\mathbb{P}^n}$ is not Nakano positive.

    Finally, if $Q$ has a singular Hermitian metric $h$ which is (globally) Nakano semi-positive and satisfying $\nu(-\log\mathrm{det}\,h,x)<2$ for any point $x\in \mathbb{P}^n$, then from the vanishing theorem (see [Wat22b,\,Theorem\,6.1]) for singular Nakano semi-positivity we have 
    \begin{align*}
        H^q(\mathbb{P}^n,K_{\mathbb{P}^n}\otimes\mathcal{O}_{\mathbb{P}^n}(1)\otimes\mathscr{E}(h))=0
    \end{align*}
    for $q\geq1$. By the fact that if $\nu(-\log\mathrm{det}\,h,x)<2$ then $\mathscr{E}(h)=\mathcal{O}(Q)$ (see the proof of [Wat22b,\,Theorem\,6.2]), we get 
    \begin{align*}
        0=H^q(\mathbb{P}^n,K_{\mathbb{P}^n}\otimes\mathcal{O}_{\mathbb{P}^n}(1)\otimes\mathscr{E}(h))\cong H^q(\mathbb{P}^n,K_{\mathbb{P}^n}\otimes \mathrm{det}\,Q\otimes Q)
        \cong H^q(\mathbb{P}^n,K_{\mathbb{P}^n}\otimes T_{\mathbb{P}^n}).
    \end{align*}
    But this vanishing contradict that $H^{n-1}(\mathbb{P}^n,K_{\mathbb{P}^n}\otimes T_{\mathbb{P}^n})\cong \mathbb{C}$.
\end{proof}

\vspace{2mm}

$\textit{Proof of Theorem \ref{min ext prop and not Nak semi-posi}}$.
Let $I_V$ be a trivial Hermitian metric on $\underline{V}$ then $I_V$ has the minimal extension property by the optimal Ohsawa-Takegoshi $L^2$-extension theorem (see \cite{Blo13}, \cite{GZ12}).
From Proposition \ref{LS22 Prop 6 and 7}, the induced Hermitian metric $h'$ on $Q$ has the minimal extension property.
Then this theorem is shown by Proposition \ref{Q has no positivity}. \qed

\vspace{3mm}

Finally, we ascertain by concrete calculations that the naturally induced smooth metric $h_Q$ of $Q$ has indeed the minimal extension property.
Here, this metric $h_Q$ induced from $I_{\underline{V}}$ and $g$ defined by $|u|_{h_Q}:=|g^*u|_{I_V}$ for any section $u$ of $Q$.

Let $a\in\mathbb{P}^n$ be fixed. Choose an orthonormal basis $(e_0,e_1,\ldots,e_n)$ of $\mathbb{C}^{n+1}$ such that $a=[e_0]$.
Consider the natural embedding $\mathbb{C}^n \hookrightarrow \mathbb{P}^n:0\mapsto a$
which sends $z=(z_1,\ldots,z_n)\mapsto [e_0+z_1e_1+\cdots+z_ne_n]$. Then $\varepsilon(z)=e_0+z_1e_1+\cdots+z_ne_n$
defines a non-zero hol section of $\mathcal{O}_{\mathbb{P}^n}(-1)|_{\mathbb{C}^n}$.
The adjoint homomorphisms $g^*:Q\to\underline{V}$ is $C^\infty$ and can be described as the orthogonal splitting of the above exact sequence.
The images $(\tilde{e}_1,\ldots,\tilde{e}_n)$ of $(e_1,\ldots,e_n)$ in $Q$ define a local holomorphic frame of $Q|_{\mathbb{C}^n}$
and we already know that $gg^*=\mathrm{id}_{\underline{V}}$ and 
\begin{align*}
    g^*\cdot\tilde{e}_j=e_j-\frac{\langle e_j,\varepsilon\rangle}{|\varepsilon|^2}\varepsilon=e_j-\frac{\overline{z}_j}{1+|z|^2}\varepsilon=e_j-\zeta_j\varepsilon,
\end{align*}
where put $\zeta_j=\frac{\overline{z}_j}{1+|z|^2}$ (see [Dem-book,\,ChapterV]). 
By $gg^*=\mathrm{id}_{\underline{V}}$ and $\varepsilon\in\mathrm{ker}g$, we get $\tilde{e}_j=gg^*\tilde{e}_j=g(e_j-\zeta_j\varepsilon)=ge_j$.
From these, the matrix representations of $g$ and $g^*$ with respect to frames $(\tilde{e}_1,\ldots,\tilde{e}_n)$ and $(e_1,\ldots,e_n)$ is as follows.
\begin{align*}
    g=\begin{pmatrix}
        -z_1 &    &  \\
        \vdots &  I_n &  \\
        -z_n &    &  \\
    \end{pmatrix}, \quad
    g^*=\begin{pmatrix}
        0 \\
        I_n \\
        \end{pmatrix}
        +G^*, \quad
    G^*=
    \begin{pmatrix}
        -\zeta_1 & \cdots & \cdots & -\zeta_n \\
        -\zeta_1z_1 &  &  & -\zeta_nz_1 \\
        \vdots & & & \vdots \\
        -\zeta_1z_n & \cdots & \cdots & -\zeta_nz_n \\
    \end{pmatrix},       
\end{align*}
where we can write $G^*=(-\zeta_1\varepsilon,\cdots,-\zeta_n\varepsilon)$. In this setting, we prove the following.

\begin{proposition}\label{min ext prop on Q}
    There exists a smooth Hermitian metric $h_Q$ on $Q=\underline{V}/\mathcal{O}_{\mathbb{P}^n}(-1)$ such that $h_Q$ has the minimal extension property.
\end{proposition}

\begin{proof}
    Let $I_V$ be a trivial Hermitian metric on $\underline{V}$ then $I_V$ has the minimal extension property by the optimal Ohsawa-Takegoshi theorem.
    We define the natural smooth Hermitian metric $h_Q$ of $Q$ induced from $I_V$ by $|u|_{h_Q}:=|g^*u|_{I_V}$ for any section $u$ of $Q$.
    We show that $h_Q$ has the minimal extension property. By the minimal extension property of $I_V$, for any $a\in\mathbb{P}^n$ and any $v\in Q_a$ with $|v|_{h_Q}=|g^*v|_{I_V}=1$, there is a holomorphic section $s\in H^0(B,\underline{V})$ such that 
    \begin{align*}
        s(0)=g^*v \quad \mathrm{and} \quad \frac{1}{\mathrm{Vol}(B)}\int_B|s|^2_{I_V}dV_B\leq1,
    \end{align*}
    where $g^*v\in \underline{V}_a$. From $gg^*=\mathrm{id}_{\underline{V}}$ then the composition $gs$ is a holomorphic section, i.e. $gs\in H^0(B,Q)$, and $gs(0)=gg^*v=v$.
    Hence, if $|gs|^2_{h_Q}=|g^*gs|^2_{I_V}\leq|s|^2_{I_V}$ on $B$ then $h_Q$ has the minimal extension property.

    We can write $s=\sum^n_{j=0}s_je_j=\sigma_0\varepsilon+\sum^n_{j=0}\sigma_je_j\in H^0(B,\underline{V})$, where $\sigma_0=s_0, \,\sigma_j=s_j-s_0z_j$ and $s_j\in\mathcal{O}(B)$.
    Then we have that 
    \begin{align*}
        gs&=\sum^n_{j=1}\sigma_jge_j=\sum^n_{j=1}\sigma_j\widetilde{e}_j,\\    
        g^*gs&=\sum^n_{j=1}\sigma_jg^*\widetilde{e}_j=\sum^n_{j=1}\sigma_j(e_j-\zeta_j\varepsilon)
        =\Bigl(\sum^n_{j=1}\zeta_j\sigma_j\Bigr)e_0+\sum^n_{j=1}\Bigl(\sigma_j-z_j\sum^n_{k=1}\zeta_k\sigma_k\Bigr)e_j,\\
        |g^*gs|^2_{I_V}&=\left|\sum^n_{j=1}\zeta_j\sigma_j\right|^2+\sum^n_{j=1}\left|\sigma_j-z_j\sum^n_{k=1}\zeta_k\sigma_k\right|^2\\
        &\leq(1+|z|^2)\left|\sum^n_{j=1}\zeta_j\sigma_j\right|^2+\sum^n_{j=1}|\sigma_j|^2=\frac{1}{1+|z|^2}\left|\sum^n_{j=1}\overline{z}_j\sigma_j\right|^2+\sum^n_{j=1}|\sigma_j|^2\\
        &\leq\frac{1}{1+|z|^2}\Bigl(\left|\sum^n_{j=1}s_j\overline{z}_j\right|^2+|s_0|^2|z|^4\Bigr)+\sum^n_{j=1}\Bigl(|s_j|^2+|s_0|^2|z_j|^2\Bigr)\\
        &=\frac{|z|^2}{1+|z|^2}\Bigl(\sum^n_{j=1}|s_j|^2+|s_0|^2|z|^2\Bigr)+\sum^n_{j=1}|s_j|^2+|s_0|^2|z|^2\\
        &=\frac{1+2|z|^2}{1+|z|^2}\Bigl(\sum^n_{j=1}|s_j|^2+|s_0|^2|z|^2\Bigr),
    \end{align*} 
    where $\sum^n_{j=1}\overline{z}_j\sigma_j=\sum^n_{j=1}(s_j\overline{z}_j-s_0|z_j|^2)=\sum^n_{j=1}s_j\overline{z}_j-s_0|z|^2$.
    Therefore, if 
    \begin{align*}
        \frac{1+2|z|^2}{1+|z|^2}\Bigl(\sum^n_{j=1}|s_j|^2+|s_0|^2|z|^2\Bigr)\leq|s|^2_{I_V}=\sum^n_{j=0}|s_j|^2,
    \end{align*}
    i.e. $|z|^2\sum^n_{j=1}|s_j|^2\leq|s_0|^2(1-2|z|^4)$, then we obtain $|g^*gs|^2_{I_V}\leq|s|^2_{I_V}$.

    Here, $s_j$ is expressed as a scalar multiple of $s_0$ for any $j$. In fact, by the optimal Ohsawa-Takegoshi extension theorem for trivial line bundle, there is a holomorphic function $f\in\mathcal{O}(B)$ such that 
    \begin{align*}
        f(0)=1 \quad \mathrm{and} \quad \frac{1}{\mathrm{Vol}(B)}\int_B|f|^2dV_B\leq1.
    \end{align*}
    We write $g^*v=\sum^n_{j=0}w_je_j\in\underline{V_a}$ where $1=|g^*v|^2_{I_V}=\sum^n_{j=0}|w_j|^2$ and $w_j\in\mathbb{C}$. 
    By changing the subscript of the local trivial frame $(e_j)$, $w_0\ne0$ can be assumed. Therefore, we can take $s_j:=w_jf=\frac{w_j}{w_0}f\in\mathcal{O}(B)$.
    Indeed, it is $s(0)=\sum^n_{j=0}w_jf(0)e_j=\sum^n_{j=0}w_je_j=g^*v$ and $|s|^2_{I_V}=(\sum^n_{j=0}|w_j|^2)|f|^2=|f|^2$.

    Thus the condition $|z|^2\sum^n_{j=1}|s_j|^2\leq|s_0|^2(1-2|z|^4)$ is sufficient for $2|z|^2+|z|^2(1-1/|w_0|^2)-1\leq0$.
    Since $w_0$ is taken as one of the non-zero in $\{w_0,\ldots,w_n\}$ that satisfy $\sum^n_{j=0}|w_j|^2=1$, we get $|w_0|^2\geq\frac{1}{1+n}$.
    Hence, if the radius of $B$ is taken to be smaller than $(-n+\sqrt{n^2+8})/4>0$ which is a solution of $2r^2+nr-1=0$, 
    then we have that $|gs|^2_{h_Q}=|g^*gs|^2_{I_V}\leq|s|^2_{I_V}$ on $B$ for any solution $s$ of the optimal Ohsawa-Takegoshi extension theorem for any $g^*v\in\underline{V}_a$.
\end{proof}

{\bf Acknowledgement. } 
The author would like to thank my supervisor Professor Shigeharu Takayama for guidance and helpful advice, and Professor Takahiro Inayama for useful advice.
The author also thanks Yoshiaki Suzuki for the helpful discussion about Kodaira-Spencer maps.

$ $

\rightline{\begin{tabular}{c}
    $\it{Yuta~Watanabe}$ \\
    $\it{Graduate~School~of~Mathematical~Sciences}$ \\
    $\it{The~University~of~Tokyo}$ \\
    $3$-$8$-$1$ $\it{Komaba, ~Meguro}$-$\it{ku}$ \\
    $\it{Tokyo, ~Japan}$ \\
    ($E$-$mail$ $address$: watayu@g.ecc.u-tokyo.ac.jp)
\end{tabular}}

\end{document}